\documentclass[12pt]{amsart}
\usepackage[colorlinks=true, pdfstartview=FitV, linkcolor=blue, citecolor=blue, urlcolor=blue, breaklinks=true]{hyperref}
\usepackage{amsmath,amsfonts,amssymb,amsthm,amscd,array,comment,
euscript,mathtools,etoolbox,latexsym,stmaryrd,mathrsfs,mathdesign,mathrsfs,
thmtools,lipsum}
\usepackage[utf8]{inputenc}
\usepackage{paralist}
\usepackage{placeins}
\usepackage[usenames]{color}
\usepackage{booktabs}
\usepackage[all]{xy}
\usepackage{tikz-cd} 
\SelectTips{cm}{10}     
  \everyxy={<2.5em,0em>:} 

\DeclareMathAlphabet{\mathpzc}{OT1}{pzc}{m}{it}

\newtoggle{comments}
\newtoggle{details}
\newtoggle{detailsnote}

\toggletrue{comments}      
\togglefalse{comments}     

\toggletrue{details}       
\togglefalse{details}      

\toggletrue{detailsnote}   

\newcommand\C{\mathbb{C}}
\newcommand\Z{\mathbb{Z}}
\newcommand\Q{\mathbb{Q}}
\newcommand\R{\mathbb{R}}
\newcommand\N{\mathbb{N}}

\newcommand\g{\mathfrak{g}}

\newcommand\fs{\mathfrak{s}}
\newcommand\fp{\mathfrak{p}}
\newcommand\fq{\mathfrak{q}}
\newcommand\fh{\mathfrak{h}}
\newcommand\fn{\mathfrak{n}}

\newcommand\fb{\mathfrak{b}}
\newcommand\fsl{\mathfrak{sl}}
\newcommand\fsp{\mathfrak{sp}}
\newcommand\fosp{\mathfrak{osp}}

\newcommand\fo{\mathfrak{o}}
\newcommand\fgl{\mathfrak{gl}}

\newcommand\fu{\mathfrak{u}}
\newcommand\fk{\mathfrak{k}}

\newcommand\sA{\mathsf{A}}

\newcommand\D{\Delta}

\newcommand\cA{\mathcal{A}}
\newcommand\cF{\mathcal{F}}
\newcommand\cC{\mathcal{C}}

\newcommand\cB{\mathcal{B}}
\newcommand\cL{\mathcal{L}}

\newcommand\cS{\mathcal{S}}

\newcommand\cD{\mathcal{D}}
\newcommand\cW{\mathcal{W}}
\newcommand\cO{\mathcal{O}}

\newcommand\bx{\mathbf{x}}

\newcommand\bV{\mathbf{V}}

\newcommand\bsp{\mathbf{sp}}

\newcommand\bU{\mathbf{U}}

\newcommand\bL{\mathbf{L}}

\DeclareMathOperator{\im}{im} 
\DeclareMathOperator{\Hom}{Hom}

\DeclareMathOperator{\Ext}{Ext}

\DeclareMathOperator{\End}{End}

\DeclareMathOperator{\Int}{Int}

\DeclareMathOperator{\Span}{Span}

\DeclareMathOperator{\Id}{Id}

\DeclareMathOperator{\Supp}{Supp} 

\DeclareMathOperator{\Ind}{Ind}

\DeclareMathOperator{\Ker}{Ker}

\theoremstyle{plain}
\newtheorem{theo}{Theorem}[section]
\newtheorem*{theo*}{Theorem}
\newtheorem{prop}[theo]{Proposition}
\newtheorem{lem}[theo]{Lemma}
\newtheorem{cor}[theo]{Corollary}

\theoremstyle{definition}

\newtheorem*{rem*}{Remark}
\newtheorem{rem}[theo]{Remark}

\newtheorem{example}[theo]{Example}

\numberwithin{equation}{section}

\allowdisplaybreaks

\newcommand\xqed[1]{%
  \leavevmode\unskip\penalty9999 \hbox{}\nobreak\hfill
  \quad\hbox{#1}}
\newcommand\demo{\xqed{$\blacksquare$}}

\iftoggle{comments}{
\newcommand{\comments}[1]{ \begin{center} \parbox{5 in}{{\bf {\footnotesize Comments:  }}{\footnotesize \textit{#1}}} \end{center}}
}{
\newcommand{\comments}[1]{}
}

\iftoggle{details}{
\newcommand{\details}[1]{\smallskip \color{blue} \begin{footnotesize} \textbf{Details:} #1 \end{footnotesize} \color{black}}
}{
\newcommand{\details}[1]{}
}

\begin{document}
%
\title[Integrable bounded weight modules of Lie superalgebras at infinity]{Integrable bounded weight modules of classical Lie superalgebras at infinity}
\author{Lucas Calixto}
\address{L.~Calixto: Department of Mathematics, Federal University of Minas Gerais, Belo Horizonte, Brazil}
\email{lhcalixto@ufmg.br}

\author{Ivan Penkov}
\address{I.~Penkov: Jacobs University Bremen Campus Ring 1, 28759 Bremen, Germany}
\email{i.penkov@jacobs-university.de}





\begin{abstract}
We classify integrable bounded simple weight modules over classical Lie superalgebras at infinity. We also study the categories of such modules, and we prove that for most of the classical Lie superalgebras at infinity the respective category is semisimple.
\end{abstract}

\maketitle \thispagestyle{empty}


%
\section{Introduction}
In the last decade there has been an active study of various categories of modules over finitary simple Lie algebras. Over the field of complex numbers $\C$, up to isomorphism there are three such Lie algebras: $\fsl(\infty)$, $\fsp(\infty)$ and $\fo(\infty)$ \cite{Bar99}. In \cite{PS11, D-CPS16} categories of integrable modules have been studied. More recently, in \cite{Nam17, CP19, PS19}, analogs of the category $\cO$ have been investigated. 

In \cite{Ser11}, V. Serganova has demonstrated that passing to the super setting is very useful. In particular, she showed that the equivalence of the categories of tensor modules over $\fo(\infty)$ and $\fsp(\infty)$, discovered in \cite{SS15} and \cite{D-CPS16}, admits a natural explanation in terms of the category of tensor modules over the Lie superalgebra $\fosp(\infty | \infty)$. Moreover, this latter category turns out to be equivalent to both former categories.

Motivated by this, we decided to study the extension, to the Lie superalgebra setting, of the recent classification of integrable bounded simple weight modules of $\fsl(\infty)$, $\fsp(\infty)$ and $\fo(\infty)$ obtained in \cite{GP20}. The Lie superalgebras we consider are listed in Table~\ref{table1} below. Beyond the technical challenge of classifying integrable bounded simple weight modules over these Lie superalgebras, we have been interested in the respective categories of integrable bounded weight modules. Over finitary Lie algebras, the category of bounded weight modules is semisimple due to an extension of Hermann Weyl's semisimplicity theorem proved by the second author and V. Serganova in \cite{PS11}. It is natural to ask whether semisimplicity holds also in the superalgebra case. We show that the categories of integrable bounded weight modules are indeed semisimple for all superalgebras $\g$ we consider, except for $\g$ isomorphic to $\fsl(\infty | 1)$ or to $\fq(\infty)$ where the category is ``almost'' semisimple. This semisimplicity result shows how special integrable bounded weight modules are.

The paper is organized as follows. In Section~\ref{sec:prel} we give some relevant basic definitions. In Section~\ref{sec:int.bound.gl} we discuss the classification of integrable bounded simple weight modules of the Lie algebra $\fgl(\infty)$. Our main classification result is presented in Section~\ref{sec:classification}. The categories of integrable bounded weight modules for the various Lie superalgebras $\g$ are discussed in Section~\ref{sec:blocks}. Finally, in the Appendix, we discuss the $\Ext$'s in the category of weight modules and provide a sufficient condition for splitting of extensions of locally simple $\g$-modules in a more general setting.

\medskip

\paragraph{\textbf{Notation}} Set $\Z^\times := \Z\setminus \{0\}$. All vector spaces, algebras, and tensor products are defined over the field of complex numbers $\C$, unless otherwise stated. The superscript $^*$ always indicates dual vector space. For any Lie superalgebra $\fk = \fk_{\bar 0}\oplus \fk_{\bar 1}$, set $\fk':=[\fk, \fk]$, and denote by $\bU(\fk)$ the universal enveloping algebra of $\fk$. If $T\subseteq \bU(\fk)$ is a subset, then we let $C_{\bU(\fk)}(T) := \{x\in \bU(\fk)\mid [x,T] = 0\}$ denote the centralizer of $T$ in $\bU(\fk)$. The symbol $\subsetplus$ (or $\supsetplus$) stands for semidirect sum of Lie superalgebras, the round side pointing toward the ideal. By $\langle \cdot \rangle_R$ we denote span over a ring $R$. If $M = M_{\bar 0}\oplus M_{\bar 1}$ is a $\Z_2$-graded vector space, then $\Pi M$ is the space with changed parity, i.e., $(\Pi M)_{\bar 0} = M_{\bar 1}$ and $(\Pi M)_{\bar 1} = M_{\bar 0}$. The parity of a homogeneous vector $v\in M$ will be denoted by $|v|\in \Z_2$, and the dimension of $M$ is denoted by $\dim M_{\bar 0}|\dim M_{\bar 1}$. Unless otherwise stated, by  homomorphisms of $\Z_{2}$-graded vector spaces we mean linear transformations that preserve parity.  For $a\in \Z_{>0}$, the $a$-th symmetric and exterior powers of a $\Z_2$-graded vector space $M$ are given, respectively, by
	\[
S^a M:=\bigoplus_{i+j=a} {\sf S}^iM_{\bar 0}\otimes {\sf \Lambda}^jM_{\bar 1},\quad \Lambda^a M:=\bigoplus_{i+j=a}{\sf \Lambda}^i M_{\bar 0}\otimes {\sf S}^jM_{\bar 1},
	\]
where ${\sf S}^i$ and ${\sf \Lambda}^i$ denote the usual $i$-th symmetric and exterior powers of a vector space.

\iftoggle{detailsnote}{
\medskip

}{}

\medskip

\paragraph{\textbf{Acknowledgments}} I.P. has been supported in part by  DFG grant PE 980-7/1. L.C. was supported by CAPES grant 88881.119190/2016-01. L.C. acknowledges the hospitality of Jacobs University. The authors would like to thank Vera Serganova for helpful discussions and a referee for the thorough reading of our paper.

\section{Preliminaries}\label{sec:prel}

Throughout the paper we denote by $\g = \varinjlim \g(n)$ one of the  Lie superalgebras defined as the direct limit of the following embeddings $\g(n) \hookrightarrow \g(n+1)$:
\begin{enumerate}\label{def:dir.lim.superalg.}
\item \label{def:dir.lim.superalg.A1} $\fsl(\infty|m) : \fsl(n|m)\hookrightarrow \fsl(n+1|m)$;
\item \label{def:dir.lim.superalg.A2} $\fsl(\infty|\infty) : \fsl(n+1|n)\hookrightarrow \fsl(n+2|n+1)$;
\item \label{def:dir.lim.superalg.B1} $\fosp_B(\infty|2k) : \fosp(2n+1|2k)\hookrightarrow \fosp(2n+3|2k)$;
\item \label{def:dir.lim.superalg.B2} $\fosp_B(\infty|\infty) : \fosp(2n+1|2n)\hookrightarrow \fosp(2n+3|2n+2)$;
\item \label{def:dir.lim.superalg.B3} $\fosp_B(m|\infty) : \fosp(m|2n)\hookrightarrow \fosp(m|2n+2)$, for $m$ odd;
\item \label{def:dir.lim.superalg.C} $\fosp_C(2|\infty) : \fosp(2|2n)\hookrightarrow \fosp(2|2n+2)$;
\item \label{def:dir.lim.superalg.D1} $\fosp_D(\infty|2k) : \fosp(2n|2k)\hookrightarrow \fosp(2n+2|2k)$;
\item \label{def:dir.lim.superalg.D2} $\fosp_D(\infty|\infty) : \fosp(2n|2n)\hookrightarrow \fosp(2n+2|2n+2)$;
\item \label{def:dir.lim.superalg.D3} $\fosp_D(m|\infty) : \fosp(m|2n)\hookrightarrow \fosp(m|2n+2)$, for $m$ even, $m\neq 2$;
\item \label{def:dir.lim.superalg.p} $\bsp(\infty) : \bsp(n)\hookrightarrow \bsp(n+1)$;
\item \label{def:dir.lim.superalg.q} $\fq(\infty) : \fq(n)\hookrightarrow \fq(n+1)$,
\end{enumerate}
see \cite{Pen04} for details. The first two embeddings are given respectively by
\begin{equation}\label{eq:emb.block.matrices}
\left(\begin{array}{c|c}
    A & B  \\
\hline
    C & D
  \end{array}\right) \mapsto
\left(\begin{array}{cr|l}
	0 & {\bf 0} & {\bf 0} \\
    {\bf 0} & A & B  \\
\hline
    {\bf 0} &  C & D
  \end{array}\right)\quad\text{and}\quad \left(\begin{array}{c|c}
    A & B  \\
\hline
    C & D
  \end{array}\right) \mapsto
\left(\begin{array}{cr|lc}
	0 & {\bf 0} & {\bf 0} &0 \\
    {\bf 0} & A & B & {\bf 0} \\
\hline
    {\bf 0} &  C & D & {\bf 0}\\
    0 & {\bf 0} & {\bf 0} & 0
  \end{array}\right),
\end{equation}
where the matrices ${\bf 0}$ are assumed to be of the appropriate size. The embeddings in \eqref{def:dir.lim.superalg.A1}-\eqref{def:dir.lim.superalg.q} are respective restrictions of the embeddings in  \eqref{eq:emb.block.matrices}.  If $\g$ is given by \eqref{def:dir.lim.superalg.A1} or \eqref{def:dir.lim.superalg.A2}, then $\g$ is of \emph{type $A$}; if $\g$ is given by \eqref{def:dir.lim.superalg.B1}, \eqref{def:dir.lim.superalg.B2} or \eqref{def:dir.lim.superalg.B3}, then $\g$ is of \emph{type $B$}; if $\g$ is given by  \eqref{def:dir.lim.superalg.D1}, \eqref{def:dir.lim.superalg.D2} or \eqref{def:dir.lim.superalg.D3}, then $\g$ is of \emph{type $D$}. In all cases except \eqref{def:dir.lim.superalg.q}, $\g$ admits a $\Z$-grading $\g = \bigoplus_{i\in \Z}\g_i$ compatible with the $\Z_2$-grading, i.e. $\g_{\bar 0} = \bigoplus_{2i}\g_i$ and $\g_{\bar 1} = \bigoplus_{2i+1}\g_i$. Table~\ref{table1} shows explicitly the Lie algebras $\g_{\bar 0}$ and $\g_0$. We refer to \cite[Tables on Lie superalgebras, page 342]{FSS00} for a description of $\g(n)_{\bar 0}$ and $\g(n)_0$.

We point out that the pairs $(\fosp_B(\infty|2k), \fosp_D(\infty|2k))$ and $(\fosp_B(\infty|\infty), \fosp_D(\infty|\infty))$ are pairs of isomorphic Lie superalgebras. The reader will check this using the well known fact that the Lie algebras $\fo_B(\infty) := \varinjlim \fo(2n+1)$ and $\fo_D(\infty) := \varinjlim \fo(2n)$ are isomorphic. However, in this paper we consider the Lie superalgebras in a pair separately, as we equip them (see the next section) with non-conjugate Cartan subalgebras. This makes the Lie superalgebras in a pair "different" from the point of view of weight modules.

\begin{table}
\begin{tabular}{c  c c}
\toprule
$\g$ & $\g_{\bar 0}$ & $\g_0$ \\
\midrule
$\fsl(\infty|m)$ & $\fgl(\infty)\oplus \fsl(m)$ & $\fgl(\infty)\oplus \fsl(m)$ \\
$\fsl(\infty | \infty)$ & $\fgl(\infty)\oplus \fsl(\infty)$ & $\fgl(\infty)\oplus \fsl(\infty)$ \\
$\fosp_B(\infty|2k)$ & $\fo_B(\infty)\oplus \fsp(2k)$ & $\fo_B(\infty)\oplus \fgl(k)$ \\
$\fosp_B(\infty|\infty)$ & $\fo_B(\infty)\oplus \fsp(\infty)$ & $\fo_B(\infty)\oplus \fgl(\infty)$ \\
$\fosp_B(m|\infty)$, $m$ odd & $\fo(m)\oplus \fsp(\infty)$ & $\fo(m)\oplus \fgl(\infty)$ \\
$\fosp_C(2|\infty)$ & $\C\oplus \fsp(\infty)$ & $\C\oplus \fsp(\infty)$ \\
$\fosp_D(\infty|2k)$ & $\fo_D(\infty)\oplus \fsp(2k)$ & $\fo_D(\infty)\oplus \fgl(k)$ \\
$\fosp_D(\infty|\infty)$ & $\fo_D(\infty)\oplus \fsp(\infty)$ & $\fo_D(\infty)\oplus \fgl(\infty)$ \\
$\fosp_D(m|\infty)$, $m$ even, $m\neq 2$ & $\fo(m)\oplus \fsp(\infty)$ & $\fo(m)\oplus \fgl(\infty)$ \\
$\bsp(\infty)$ & $\fsl(\infty)$ & $\fsl(\infty)$ \\
$\fq(\infty)$ & $\fgl(\infty)$ & \\
\bottomrule
\smallskip
\end{tabular}
\caption{Classical Lie superalgebras at infinity, their even part and their 0-th degree component} \label{table1}
\end{table}

\subsection{Generalities on $\g$-modules}
We call a $\g$-module $M$ \emph{integrable} if for every $m\in M$, $g\in \g$ one has 
	\[
\dim \langle m, gm, g^2m, \ldots\rangle_\C<\infty.
	\]

Let $\fh\subseteq \g$ denote the \emph{splitting Cartan subalgebra} of diagonal matrices in the Lie algebra $\g_{\bar 0}$ \cite{D-CPS07}. In other words, $\fh$ is the direct limit of the diagonal Cartan subalgebras of the Lie algebras $\g(n)_{\bar 0}$ under the fixed embeddings $\g(n)_{\bar 0}\hookrightarrow  \g(n+1)_{\bar 0}$. A $\g$-module $M$ is a \emph{weight module }(with respect to $\fh$) if
	\[
M = \bigoplus_{\lambda\in \fh^*} M^\lambda,
	\]
where $M^\lambda := \{m\in M\mid hm=\lambda(h)m,\ \forall \ h\in \fh\}$. The \emph{support} of $M$ is the set $\Supp M := \{\lambda\in \fh^*\mid M^\lambda\neq 0\}\subseteq \fh^*$. The elements $\lambda\in \Supp M$ are the \emph{weights of} $M$, and nonzero vectors in $M^\lambda$ are called \emph{weight vectors} of weight $\lambda$. A weight module $M$ is said to be \emph{bounded} if there exists $k\in \Z_{>0}$ such that $\dim M^\lambda \leq k$ for all $\lambda\in \Supp M$.

Under the adjoint action of $\fh$ on $\g$ we have the decomposition
	\[
\g = \g^0\oplus \bigoplus_{\alpha\in \D} \g^\alpha,
	\]
where $\g^0 = \fh$ if $\g\ncong \fq(\infty)$, and $\D := \Supp \g\setminus \{0\}$. The elements of $\D$ are the \emph{roots} of $\g$, and $\D$ is the \emph{root system} of $\g$.  To describe $\D$, we note first that $\g \subseteq \fgl(\infty | \infty) = \varinjlim \fgl(n|n)$ and that matrices in $\fgl(\infty | \infty)$ are indexed by $\Z^\times \times \Z^\times $, where $(0,0)$ is identified with the intersection of the two orthogonal lines that separate the blocks of the matrices in $\fgl(\infty | \infty)$. Let $E_{i,j}\in \fgl(\infty | \infty)$ denote the elementary matrix with entry $1$ at position $(i,j)$ and zeros elsewhere. For any $i\in \Z^\times $ we let $\varepsilon_i\in \fh^*$ be the linear functional defined by $\varepsilon_i(E_{j,j}) = \delta_{i, j}$ for all $j\in \Z^\times $, and we set $\delta_i:=\varepsilon_{-i}$, for any $i\in \Z_{>0}$. We should point out that the $\varepsilon_i$'s could be indexed by an arbitrary countable set, not necessarily $\Z^\times $. We fix $\Z^\times$ for convenience.

The root system of $\g$ is given as follows:
{\footnotesize{\begin{align*}
& \fsl(\infty | m) : \D = \{\varepsilon_i - \varepsilon_j,\ \delta_r - \delta_s,\ \pm(\varepsilon_i - \delta_r)\mid i,j\in \Z_{>0}\cap [0,m],\ r,s\in \Z_{>0}\}; \\
& \fsl(\infty|\infty) : \D = \{\varepsilon_i - \varepsilon_j,\ \delta_r - \delta_s,\ \pm(\varepsilon_i - \delta_r)\mid i,j\in \Z_{>0},\ r,s\in \Z_{>0}\}; \\
& \fosp_B(\infty|2k) :  \D = \{\pm \varepsilon_i \pm \varepsilon_j,\ \pm \delta_r \pm \delta_s,\ \pm 2 \varepsilon_i,\ \pm\delta_r,\ \pm \varepsilon_i \pm \delta_r,\ \pm \varepsilon_i \mid i,j\in \Z_{>0}\cap [0,k],\ r,s\in \Z_{>0}\}; \\
& \fosp_B(\infty|\infty) : \D = \{\pm \varepsilon_i \pm \varepsilon_j,\ \pm \delta_r \pm \delta_s,\ \pm 2 \varepsilon_i,\ \pm\delta_r,\ \pm \varepsilon_i \pm \delta_r,\ \pm \varepsilon_i \mid i,j\in \Z_{>0},\ r,s\in \Z_{>0}\}; \\
& \fosp_B(m|\infty) : \D = \{\pm \varepsilon_i \pm \varepsilon_j,\ \pm \delta_r \pm \delta_s,\ \pm 2 \varepsilon_i,\ \pm\delta_r,\ \pm \varepsilon_i \pm \delta_r,\ \pm \varepsilon_i \mid  i,j\in \Z_{>0},\ r,s\in \Z_{>0}\cap [-m,0]\}; \\
& \fosp_C(2|\infty) : \D = \{\pm \varepsilon_i \pm \varepsilon_j,\ \pm 2 \epsilon_i,\ \pm \varepsilon_i \pm \delta_1 \mid i,j\in \Z_{>0}\}; \\
& \fosp_D(\infty|2k) : \D = \{\pm \varepsilon_i \pm \varepsilon_j,\ \pm \delta_r \pm \delta_s,\ \pm 2 \epsilon_i,\ \pm \varepsilon_i \pm \delta_r \mid  i,j\in \Z_{>0},\ r,s\in \Z_{>0}\cap [-k,0]\}; \\
& \fosp_D(\infty|\infty) : \D = \{\pm \varepsilon_i \pm \varepsilon_j,\ \pm \delta_r \pm \delta_s,\ \pm 2 \epsilon_i,\ \pm \varepsilon_i \pm \delta_r \mid  i,j\in \Z_{>0},\ r,s\in \Z_{>0}\}; \\
& \fosp_D(m|\infty) :  \D = \{\pm \varepsilon_i \pm \varepsilon_j,\ \pm \delta_r \pm \delta_s,\ \pm 2 \epsilon_i,\ \pm \varepsilon_i \pm \delta_r \mid  i,j\in \Z_{>0}\cap [0,m],\ r,s\in \Z_{>0}\}; \\
& \bsp(\infty) :  \D = \{\varepsilon_i - \varepsilon_j,\ -\varepsilon_i - \varepsilon_j,\ \varepsilon_i + \varepsilon_j,\ 2\varepsilon_i)\mid i,j\in \Z_{>0}\}; \\
& \fq(\infty) :  \D = \{\varepsilon_i - \varepsilon_j \mid i,j\in \Z_{>0}\}.
\end{align*}}}

If $\g\ncong \fq(\infty)$, then $\dim \g^\alpha = 1|0$ or $\dim \g^\alpha = 0|1$ for every $\alpha\in \D$. In that case, given $\pm \alpha\in \D$ we fix $X_{\pm \alpha}\in \g^{\pm \alpha}\setminus\{0\}$ so that the nonzero coordinates of $h_\alpha:=[X_\alpha, X_{-\alpha}]\in \fh$ with respect to the basis $\{E_{i,i}\mid i\in \Z^\times\}$ of the subalgebra of diagonal matrices in $\fgl(\infty|\infty)$ are equal to $1$ or $-1$. The root spaces of $\fq(\infty)$ have dimension $1|1$. In addition, here $\g^0 = \fh\oplus \fh_{\bar 1}$, and $\dim \fh_{\bar 1} = 0|\infty$. Finally, for any $\g$ and any $n\in \Z_{>0}$, we define 
	\[
\fh(n): = \fh\cap \g(n), \text{ and } \D(n):=\{\alpha\in \D\mid \g^\alpha\subseteq \g(n)\}.
	\]
	
Let $n,m\in \Z_{>0}\cup \{\infty\}$. Throughout the paper, the expression $\sum_i^n \lambda_i\delta_i + \sum_i^m \mu_i\varepsilon_i$ will be identified with the vector $(\lambda|\mu):=(\ldots, \lambda_2, \lambda_1 | \mu_1,\mu_2,\ldots)\in \C^n\times\C^m$; the vector $(\ldots, c, c | d,d,\ldots)\in \C^n\times\C^m$ with $c,d\in \C$ will be denoted by $(c^{(n)} | d^{(m)})$. Therefore, for $\g=\fgl(n|m)$ or $\g = \fosp(n|m)$ we can identify $\fh(n)^*$ with $\C^n\times \C^m$. If $\g=\fsl(n|m)$ with $n\neq m$, then we also can think of $(\lambda | \mu)\in \C^n\times\C^m$ as a weight of $\g$: we consider the image of $(\lambda | \mu)$ in $\fh(n)^*$ under the projection $(\lambda | \mu) \mapsto (\lambda | \mu) + \C (1^{(n)}|-1^{(m)})$. If $\g =\fsl(n)$ or $\g= \bsp(n)$, then we can think of $\lambda\in \C^n$ as a weight of $\g$ by considering the image of $\lambda$ in $\fh(n)^*$.

In what follows, we normalize the marks of a weight of $\fsl(n)$ in such a way that the last mark is zero. Then we have a well-defined correspondence between weights and partitions.

\subsection{Splitting Borel subalgebras} \emph{Splitting Borel subalgebras} of $\g$ are determined by \emph{triangular decompositions} of $\D$, which in turn are determined  by (non-unique) elements of $(\langle\D\rangle_{\R})^*$ (see \cite[Proposition~2]{DP98}). Namely, a given $\phi\in (\langle\D\rangle_{\R})^*$ determines the decomposition
	\[
\D = \D^-\sqcup \D^+ \text{ where } \D^\pm = \{\alpha\in \D\mid \phi(\alpha)\gtrless 0\}.
	\]	
The set $\D^+$ is called the \emph{set of positive roots} associated to $\phi$. The splitting Borel subalgebra corresponding to this decomposition is $\fb:=\fh\supsetplus \fn$, where
	\[
\fn = \bigoplus_{\alpha\in \D^+} \g^\alpha.
	\]

We now present an explicit description of splitting Borel subalgebras in terms of linear orders on countable sets. Recall that $\delta_i:=\varepsilon_{-i}$ for every $i\in \Z_{>0}$. Suppose $\g=\fsl(\infty|\infty)$. In this case, splitting Borel subalgebras of $\g$ are parameterized by linear orders $\prec$ on $\Z^\times $. More precisely, the set of positive roots corresponding to a linear order $\prec$ is 
\begin{align*}
\D(\prec) & = \{\delta_i - \delta_j\mid -i\prec -j,\ i,j\in \Z_{>0}\}\cup \{\varepsilon_i - \varepsilon_j\mid i\prec j,\ i,j\in \Z_{>0}\} \\
& \cup \{\delta_i - \varepsilon_j\mid -i\prec j,\ i,j\in \Z_{>0}\}.
\end{align*}
If $\g=\fsl(\infty|n)$ or $\fq(\infty)$, then $\Z^\times$ must be replaced respectively by $\Z_{\leq n}^\times$ and $\Z_{>0}^\times$. For $\g = \fosp_B(\infty|\infty)$ splitting Borel subalgebras of $\g$ are parameterized by pairs $(\prec, \sigma)$, where $\prec$ is a linear order on $\Z^\times $ and $\sigma$ is a map $\sigma : \Z^\times \to \{\pm 1\}$. The set of positive roots corresponding $(\prec, \sigma)$ is
\begin{align*}
\D(\prec, \sigma) & = \{\sigma(i)\delta_i - \sigma(j)\delta_j\mid -i\prec -j,\ i,j\in \Z_{>0}\}\cup \{\sigma(i)\delta_i + \sigma(j)\delta_j\mid i\neq j\in \Z_{>0}\} \\
& \cup \{\sigma(i)\varepsilon_i - \sigma(j)\varepsilon_j\mid i\prec j,\ i,j\in \Z_{>0}\}\cup \{\sigma(i)\varepsilon_i + \sigma(j)\varepsilon_j\mid i\neq j\in \Z_{>0}\} \\
& \cup  \{\sigma(i)\delta_i \mid i\in \Z_{>0}\} \cup  \{\sigma(i)\varepsilon_i \mid i\in \Z_{<0}\} 
\cup  \{\sigma(i)2\varepsilon_i \mid i\in \Z_{>0}\} \\
& \cup \{\sigma(i)\delta_i \pm  \sigma(j)\varepsilon_j\mid i\in \Z_{<0},\ j\in \Z_{>0}\}.
\end{align*}
If $\g$ is of type $\fosp_B(\infty|2k)$ or $\fosp_B(m|\infty)$, then $\Z^\times $ gets replaced respectively by $\Z_{\leq k}^\times$ and $\Z_{\geq -m}^\times$. For $\g=\fosp_D(\infty|\infty)$ the construction is analogous to that for $\fosp_B(\infty|\infty)$, however in this case we need an extra condition on $\sigma : \Z^\times \to \{\pm 1\}$: if $\prec$ admits a maximal element $i_0\in \Z_{<0}$ then $\sigma(i_0)=1$. Hence $\D(\prec, \sigma)$ is given similarly to the previous case, but now there are no roots of the form $\sigma(i)\varepsilon_i, \sigma(i)\delta_i$. If $\g$ is of type $\fosp_D(\infty|2k)$, $\fosp_D(m|\infty)$ or $\fosp_C(2|\infty)$, then $\Z^\times $ is replaced by $\Z_{\leq k}^\times$, $\Z_{\geq -m}^\times$ or $\Z_{\geq -1}^\times$, respectively. We point out that for $\fosp_C(2|\infty)$ we do not require the additional condition on the map $\sigma$. Finally, for $\g=\bsp(\infty)$ we replace $\Z^\times $ by $\Z_{>0}$ in the above discussion, and we define
\begin{align*}
\D(\prec, \sigma) & = \{\sigma(i)\varepsilon_i - \sigma(j)\varepsilon_j\mid i\prec j,\ i,j\in \Z_{>0}\}\cup \{\sigma(i)\varepsilon_i + \sigma(j)\varepsilon_j\mid i\neq j\in \Z_{>0}\}  \\
& \cup \{2\varepsilon_i \mid i\in \Z_{>0},\ \sigma(i)=1\}.
\end{align*}

The splitting Borel subalgebra corresponding to $\D(\prec)$ (respectively, $\D(\prec, \sigma)$) is denoted by $\fb(\prec)$ (respectively, $\fb(\prec, \sigma)$), and $\fn(\prec)$ (respectively, $\fn(\prec,\sigma)$) denotes its \emph{locally nilpotent radical}. Moreover, for every $n\in \Z_{>0}$, we set $\fb(\prec_n) :=\fb(\prec)\cap \g(n)$ (respectively, $\fb(\prec_n,\sigma) :=\fb(\prec,\sigma)\cap \g(n)$).

Throughout the paper, we denote by $<$ the standard order on $\Z$.

\subsection{Highest weight modules}
Let $\fb = \fh \supsetplus \fn$ be a splitting Borel subalgebra of $\g$, and $M$ be a weight module. A weight vector $0\neq v\in M^\lambda$ is a \emph{$\fb$-singular vector} if $\fn\cdot v = 0$. If $M$ is a cyclic $\g$-module generated by a $\fb$-singular vector of weight $\lambda$, we say that $M$ is a \emph{$\fb$-highest weight module}, and $\lambda$ is the \emph{$\fb$-highest weight} of $M$. Given an element $\lambda\in \fh^*$, we consider the \emph{Verma type module associated to $\lambda$ and $\fb$}
	\[
M_\fb (\lambda):=\Ind_\fb^\g U^\lambda := \bU(\g)\otimes_{\bU(\fb)} U^\lambda,
	\]
where $U^\lambda$ is a simple $\fb$-module on which $\fh$ acts via $\lambda$ and $\fn$ acts trivially. If $\g\ncong \fq(\infty)$, we require $U^\lambda$ to have dimension $1|0$. If $\g\cong \fq(\infty)$, then the dimension of $U^\lambda$ is $2^{[\# \lambda/2]}$ where $\# \lambda$ denotes the number of nonzero marks of $\lambda$, and $[a]$ denotes the greatest integer in the number $a\in \Q$. The $\g$-module $M_\fb(\lambda)$ admits a unique simple quotient which we denote by $\bL_{\fb}(\lambda)$. Accordingly, $\Pi \bL_{\fb}(\lambda)$ admits a $\fb$-highest weight space of weight $\lambda$ whose dimension is $\dim U^\lambda_{\bar 1}|\dim U^\lambda_{\bar 0}$. 

The Lie superalgebra $\g$ admits a \emph{natural module} $\bV$ with support
	\[ 
\Supp \bV = \begin{cases} 
	\{\delta_i, \varepsilon_i\} & \text{ if } \g = \fsl(\infty | \infty), \fsl(\infty | m) \\
\{\pm \delta_i, 0, \pm \varepsilon_i\} & \text{ if } \g=\fosp_B(\infty|\infty), \fosp_B(m|\infty), \fosp_B(\infty, 2k) \\
\{\pm \delta_1, \pm \varepsilon_i\} & \text{ if } \g=\fosp_C(2|\infty) \\
\{\pm \delta_i, \pm \varepsilon_i\} & \text{ if } \g=\fosp_D(\infty|\infty), \fosp_D(m|\infty), \fosp_D(\infty|2k) \\
      \{\varepsilon_i\} & \text{ if } \g=\fq(\infty) \\
\{\pm \varepsilon_i\} & \text{ if } \g=\bsp(\infty), \\
  \end{cases}
	\]
where the index $i$ runs over the respective obvious subset of $\Z^\times $. To determine $\bV$ up to isomorphism for $\g\neq \fq(\infty)$, $\bsp(\infty)$, we require that the weight spaces with weights $\delta_i$ belong to $\bV_{\bar 0}$. For $\g = \fq(\infty)$, the support determines $\bV$ up to isomorphism, and for $\g = \bsp(\infty)$ the weight spaces $\varepsilon_i$ belong to $\bV_{\bar 0}$. Furthermore, when $\g$ equals $\fsl(\infty | m)$ for $m\in \Z_{\geq 1}\cup \{\infty\}$ or $\fq(\infty)$, then $\g$ admits a \emph{conatural} module $\bV_*$ which is characterized (up to isomorphism) by the requirement that $\Supp \ (\bV_*)_z = -\Supp \bV_z$ for $z\in \Z_{2}$.

\begin{rem}
Throughout the paper, for convenience, if $\g$ is a Lie algebra we write $L_\fb(\lambda)$, $V$, and $V_*$ instead of $\bL_\fb(\lambda)$, $\bV$, and $\bV_*$, respectively.
\demo
\end{rem}

\section{Integrable bounded modules of $\fgl(\infty)$}\label{sec:int.bound.gl}

In what follows let $\fh_\fgl$ and $\fh_\fsl$ denote the Cartan subalgebras consisting of diagonal matrices in $\fgl(\infty) := \varinjlim \fgl(n)$ and $\fsl(\infty) := \varinjlim \fsl(n)$, respectively.

Let $M$ be a weight $\fsl(\infty)$-module such that $M=\bU(\fsl(\infty))\cdot m$ for some $m\in M^\lambda$, where $\lambda\in \Supp M\subseteq \fh_\fsl^*$. For any $c\in \C$ we extend $\lambda$ to an element of $\fh_{\fgl}$, which we denote also by $\lambda$, by setting $\lambda(E_{1,1}):=c$. Now we define the $\fgl(\infty)$-module $M(m,c)$ as follows: $M(m,c)$ equals $M$ as a vector space; the action of $\fsl(\infty)$ on $M(m,c)$ coincides with its action on $M$; the action of $E_{1,1}$ on $m$ is via multiplication by $c$, and, for any $u\in \bU(\fsl(\infty))^\beta$ ($\bU(\fsl(\infty))^\beta$ being a weight space of $\bU(\fsl(\infty))$ with respect to the adjoint $\fsl(\infty)$-module structure)
\begin{equation}\label{eq:extention.formula}
E_{1,1}um := (\beta+\lambda)(E_{1,1}) um = (\beta(E_{1,1})+c) um.
\end{equation}
It is easy to see that the $\mathfrak{gl}(\infty)$-module $M(m,c)$ is well defined. 

\begin{rem}\label{rem:weights.of.M(m,c)}
Notice that any element $\nu\in \Supp M \subseteq \lambda + \Z\D$ can be extended to an element of $\fh_\fgl^*$ via \eqref{eq:extention.formula}: if $\nu = \lambda+\beta$ then $\nu(E_{1,1}) = (c + \beta)(E_{1,1})$. By a slight abuse of notation, we denote such an extension also by $\nu$. Hence, $M(m,c)$ is a weight $\fgl(\infty)$-module. Moreover, since for any $\nu, \nu'\in \Supp M(m,c)$ the weight $\nu-\nu'$ lies in the root lattice of $\fgl(\infty)$ (and hence of $\fsl(\infty)$), we have $\nu\neq \nu'$ if and only if $(\nu-\nu')|_{\fh_\fsl}\neq 0$. This shows that $\Supp M(m,c)$ is obtained by extending $\Supp M$ via \eqref{eq:extention.formula}, and any two $\fh_\fgl$-weights of $M(m,c)$ are equal if and only if their corresponding restrictions to $\fh_\fsl$ are equal.
\demo
\end{rem}

Let $\fk(1)\subset \fk(2)\subset \fk(3)\cdots$ be a sequence of inclusions of Lie superalgebras, and let $\fk = \bigcup_n \fk(n)=\varinjlim \fk(n)$. A $\fk$-module $M$ is \emph{locally simple} if for each $m\in M\setminus\{0\}$ the $\fk(n)$-module $\bU(\fk(n))m$ is simple for $n\gg 0$, and $M = \bigcup_{n\gg 0} \bU(\fk(n))m$.

\begin{lem}\label{lem:loc.simpl=M(c)}
Suppose $M$ is a locally simple weight $\fgl(\infty)$-module. Then, for any $\lambda\in \Supp M|_{\fsl(\infty)} $ and $m\in (M|_{\fsl(\infty)})^\lambda\setminus\{0\}$, there is $c\in \C$ for which $M\cong M|_{\fsl(\infty)} (m,c)$.
\end{lem}
\begin{proof}
Recall that $\fgl(\infty) = \varinjlim \fgl(n)$. Set $M_\ell := \bU(\fgl(\ell))m$ for $\ell\geq 1$. Since $M_\ell$ is a simple $\fgl(\ell)$-module for $\ell\gg 0$, there is $c\in \C$ such that the action of $E_{1,1}$ on $M_\ell$ is given by \eqref{eq:extention.formula}. As $M = \bigcup_{\ell \gg 0} M_\ell$ the result follows.
\end{proof}

We recall from \cite[Proposition~4.5]{GP20} that any integrable bounded simple weight $\fsl(\infty)$-module is isomorphic to a direct limit $\varinjlim L_{\fb(<_n)}(\lambda(n))$, where, for every $n$, $\lambda(n)$ is a weight of the following types:
\begin{enumerate}
\item  $(1^{(b_n)},0^{(n-b_n)})$,
\item $(a_n,0^{(n-1)})$, 
\item  $(0^{(n-1)}, -a_n)$, 
\item $(\mu_1,\ldots,\mu_k,0^{(n-k)})$, 
\item $(0^{(n-k)}, -\mu_k,\ldots,-\mu_1)$.
\end{enumerate}
Here $B = \{b_1\leq b_2\leq \ldots\} \subseteq \Z_{>0}$ is a semi-infinite set (that is, $|B|=|\Z_{>0}\setminus B|=\infty$) satisfying $b_{n+1}\in \{b_n, b_n+1\}$, $A=\{a_1\leq a_2\leq \ldots\} \subseteq \Z_{>0}$ is an infinite set, and $\mu : = (\mu_1\geq \cdots \geq \mu_k)$ is a partition. These locally simple $\fsl(\infty)$-modules are denoted respectively by $\Lambda_B^{\frac{\infty}{2}} V$, $S_A^\infty V$, $S_A^\infty V_*$, $S^\mu V$ and $S^\mu V_*$.

Fix nonzero weight vectors:
\begin{enumerate}
\item $v_\mu\in S^\mu V$ of weight $\mu := \sum_{i=1}^k \mu_i\varepsilon_i\in \fh_\fsl^*$,
\item $v_\mu^*\in S^\mu V$ of weight $\mu^* := \sum_{i=1}^{k}-\mu_i\varepsilon_{i} \in \fh_\fsl^*$,
\item $e_A \in \Lambda^{\frac{\infty}{2}}_B V$ of weight $\varepsilon_A := \sum_{i\in A}\varepsilon_i\in \fh_\fsl^*$,
\item $v_A \in S^\infty_A V$ of weight $\lambda_A := \sum_{i\geq 1}(a_i -a_{i-1})\varepsilon_i\in \fh_\fsl^*$,
\item $v_A^* \in S^\infty_A V$ of weight $\lambda_A^* := \sum_{i\geq 1}(a_{i-1} -a_{i})\varepsilon_i\in \fh_\fsl^*$.
\end{enumerate}
Now we are ready to state the main result of this section.

\begin{theo}\label{thm:bound.s.w.gl.mod}
An integrable simple weight $\fgl(\infty)$-module $M$ is bounded if and only if $M$ is isomorphic to one of the following modules: $\Lambda_A^{\frac{\infty}{2}} V(e_A, c)$, $S_A^{\infty} V(v_A, c)$, $S_A^{\infty} V_*(v_A^*, c)$,  $S^\mu V(v_\mu, c)$, or $S^\mu V_*(v_\mu^*, c)$, where $c\in \C$ is a scalar.
\end{theo}
\begin{proof}
Set $\bU(\fgl(n))^0 := C_{\bU(\fgl(n))}(\fh_\fgl (n))$, and fix a weight $\lambda\in \Supp M$. Since $M$ is simple and bounded, Lemma~\ref{lem:A1} from the Appendix claims that the weight space $M^\lambda$ is simple as a $\bU(\fgl(n))^0$-module for $n\gg 0$. Let $m\in M^\lambda$ and let $M_n :=\bU(\fgl(n))m$. The simplicity of $M^\lambda$ as a $\bU(\fgl(n))^0$-module and the fact that $M$ is integrable imply the simplicity of $M_n$ as a $\fgl(n)$-module. Therefore, $M\cong \varinjlim_{n\gg 0} M_n$ is locally simple. Hence, by Lemma~\ref{lem:loc.simpl=M(c)} we have an isomorphism of $\fgl(\infty)$-modules $M\cong M|_{\fsl(\infty)} (m,c)$ for some $c\in \C$, and by Remark~\ref{rem:weights.of.M(m,c)} we know that $M|_{\fsl(\infty)}$ is bounded as an $\fsl(\infty)$-module.  Now the statement follows from \cite[Theorem~5.1]{GP20}.
\end{proof}

\begin{prop}\label{prop:h.w.gl(infty)}
The following statements hold.
\begin{enumerate}
\item The modules $S^\infty_A V(v_A, c)$, $S^\infty_A V_*(v_A^*, c)$ are not highest weight modules with respect to any Borel subalgebra of $\fgl(\infty)$.
\item The module $\Lambda^{\frac{\infty}{2}}_A V(e_A, c)$ is a $\fb(\prec)$-highest weight module if and only if $A\prec (\Z_{>0}\setminus A)$. In this case, we have $\Lambda^{\frac{\infty}{2}}_A V(e_A, c)\cong L_{\fb(\prec)}(\varepsilon_A)$ where $\varepsilon_A |_{\fh_\fsl} = \sum_{i\in A}\varepsilon_i$ and $\varepsilon_A$ is extended to $\fh_\fgl$ via \eqref{eq:extention.formula}.
\item The module $S^\mu V(v_\mu, c)$ (respectively, $S^\mu \bV_*(v_\mu^*, c)$) is a $\fb(\prec)$-highest weight module if and only if $i_1\prec \cdots \prec i_k\prec j$ for all $j\in \Z_{>0}\setminus \{i_1,\ldots, i_k\}$ (respectively, $i_1\succ \cdots \succ i_k\succ j$ for all $j\in \Z_{>0}\setminus \{i_1,\ldots, i_k\}$). In this case, we have $S^\mu V(v_\mu, c)\cong L_{\fb(\prec)} (\mu)$ (respectively, $S^\mu V_*(v_\mu^*, c)\cong L_{\fb(\prec)} (\mu^*)$) where $\mu |_{\fh_\fsl} = \sum_{j>0} \mu_j\varepsilon_{i_j}$ (respectively, $\mu^* |_{\fh_\fsl} = \sum_{i>0} -\mu_j\varepsilon_{i_j}$) and $\mu$ (respectively, $\mu^*$) is extended to  $\fh_\fgl^*$ via \eqref{eq:extention.formula}.
\end{enumerate}
\end{prop}
\begin{proof}
Let $\fb$ be an arbitrary splitting Borel subalgebra of $\fgl(\infty)$. The fact that a weight module $M$ is a $\fb$-highest weight $\fgl(\infty)$-module if and only if $M$ is a $\fb$-highest weight $\fsl(\infty)$-module, along with \cite[Proposition~5.2]{GP20}, implies the statement.
\end{proof}

\section{A general lemma}\label{sec:generalities.super}

In this section $\g$ is one of the Lie superalgebras introduced in Section~\ref{sec:prel}.

\begin{lem}\label{lem:ss.over.Lie}
Let $\fk$ be equal to $\g_0$ or $\g_{\bar 0}$. If $M$ is an integrable simple weight $\g$-module with finite-dimensional weight spaces, then there is an isomorphism of $\Z_2$-graded $\fk':=[\fk, \fk]$-modules
	\[
M|_{\fk'}\cong \bigoplus_i M(i),
	\]
where each $M(i)$ is an integrable simple weight $\fk'$-module with finite-dimensional weight spaces. Moreover, each $M(i)$ is also an integrable simple weight module with finite-dimensional weight spaces over $\fk$.
\end{lem}
\begin{proof}
Let $\mu$ be a weight of $M$, and consider the $\fk$-submodule $N(\mu) := \bU(\fk)M^\mu$ of $M|_{\fk'}$. Notice that the $(\fk'\cap \fh)$-weight spaces of $N(\mu)$ coincide with its $\fh$-weight spaces. Indeed, the reason is basically the same as in Remark~\ref{rem:weights.of.M(m,c)}: since $\lambda - \lambda'$ is an element of the root lattice of $\fk'$ for any two $\fh$-weights $\lambda$, $\lambda'$ of $N(\mu)$, we have $\lambda\neq \lambda'$ if and only if $(\lambda - \lambda') |_{\fh\cap \fk'} \neq 0$. Thus $N(\mu)^{\nu|_{\fh\cap \fk'}} = N(\mu)^\nu\subseteq M^\nu$ for any $\nu\in \Supp N(\mu)$, which implies that, as a $\fk'$-module, $N(\mu)$ has finite-dimensional weight spaces.

As $M|_{\fk'}$ is obviously integrable as a $\fk'$-module, so is $N(\mu)$. Then  we can use \cite[Theorem~3.7]{PS11} to conclude that each $N(\mu)$, and hence also $M|_{\fk'} = \sum_{\mu\in \Supp M} N(\mu)$ (by the general result \cite[Chapter XVII, Lemma~2.1]{Lan02}), can be written as a direct sum $\bigoplus_i M(i)$, where each $M(i)$ is an integrable simple weight $\fk'$-module with finite-dimensional weight spaces.  This proves the first statement. The second statement follows from the fact that the $(\fk'\cap \fh)$-weight spaces of each $M(i)$ are also $\fh$-weight spaces.
\end{proof}

\section{Classification results}\label{sec:classification}
\subsection{Type $A$}\label{sec:classification.type.A}
In this section 
	\[
\g = \fsl(\infty | m) \text{ for } m\in\Z_{\geq 1}\cup\{\infty\}.
	\]
	
Recall from Theorem~\ref{thm:bound.s.w.gl.mod} that any integrable bounded simple weight $\fgl(\infty)$-module is isomorphic to $M(m,c)$, for some integrable bounded simple weight $\fsl(\infty)$-module $M$, some fixed weight vector $m\in M$, and some scalar $c\in \C$. Moreover, by Remark~\ref{rem:weights.of.M(m,c)}, we know that $\Supp M(m,c)$ is obtained by extending $\Supp M$ via \eqref{eq:extention.formula}. In particular, if $\lambda = (\lambda_1, \lambda_2,\ldots)\in \C^\infty$ is in $\Supp M$ then its extension through \eqref{eq:extention.formula} to an element of $\fh_\fgl^*$ will be of the form $\lambda^d:= \lambda+((d-\lambda_1)^{(\infty)}) \in \C^\infty$, for some $d\in c + \Z$.

Consider now the isomorphism of Lie algebras $\fgl(\infty)\oplus \fsl(m) \to \fsl(\infty|m)_0$ such that
	\[
(A,B)\mapsto \left(\begin{array}{c|c}
    A & 0  \\
\hline
    0 & B
  \end{array}\right),\quad E_{1,1}\mapsto h_{\delta_1 - \varepsilon_1}:=\left(\begin{array}{c|c}
    E_{-1,-1} & 0  \\
\hline
    0 & E_{1,1}
  \end{array}\right),
	\]
where $A\in \fsl(\infty)$ and $B\in \fsl(m)$. This isomorphism induces the following correspondence of weights:
\begin{multline*}
\fh_{\fgl}^*\times \fh_\fsl^*\ni (\ldots, (\lambda_3- \lambda_1) +c, (\lambda_2- \lambda_1) + c, c)\times (\nu_1, \nu_2, \ldots) \\
\leftrightarrow (\ldots, (\lambda_3- \lambda_1) +c, (\lambda_2- \lambda_1) + c, c | 0, \nu_2 - \nu_1, \nu_3 - \nu_1,\ldots) := (\lambda^c | \nu)\in \fh^*.
\end{multline*}

By Lemma~\ref{lem:ss.over.Lie}, for an integrable bounded simple $\g$-module $M$ we have an isomorphism of $\g_0$-modules
	\[
M|_{\g_0}\cong \bigoplus_i M(i),
	\]
where each $M(i)$ is an integrable bounded simple weight $\g_0$-module. For the rest of this section we fix such a decomposition of $M|_{\g_0}$.

Recall that $m\in \Z_{\geq 1}\cup \{\infty\}$. In order to consider both cases $m<\infty$ and $m=\infty$, simultaneously, we define, for every $n\in \Z_{\geq 2}$, the elements
	\[ 
x_n := \begin{cases} 
	m & \text{ if } m\in \Z_{\geq 1} \\
n-1 & \text{ if } m=\infty.
  \end{cases}
	\]
In particular, we have  
	\[
\fsl(\infty|m) \cong \varinjlim (\g(n):=\fsl(n|x_n)).
	\]

Recall that  (unless otherwise stated) by homomorphisms of $\Z_{2}$-graded vector spaces we mean linear transformations that preserve parity.

\subsection*{The modules $S_\cA^\infty \bV$, $S_\cA^\infty \bV_*$, $\Lambda_\cA^\infty \bV$ and $\Lambda_\cA^\infty \bV_*$} By $\bV_{n}$ we denote the natural $\g(n)$-module, and by $\bV_{n}^*$ its dual. For $a,b\in \Z_{>0}$ with $b\leq a$, it is easy to check that there are unique (up to scalar) embeddings of $\g(n-1)$-modules $S^b\bV_{n-1}\hookrightarrow S^a\bV_n$, $\Lambda^b\bV_{n-1}\hookrightarrow \Lambda^a\bV_n$, and respectively, $\Pi S^b\bV_{n-1}\hookrightarrow \Pi S^a\bV_n$, $\Pi \Lambda^b\bV_{n-1}\hookrightarrow \Pi \Lambda^a\bV_n$. If $b<a$ and $x_{n-1}<x_n$, then we also have unique (up to scalar) embeddings of $\g(n-1)$-modules $S^b\bV_{n-1}\hookrightarrow \Pi S^a\bV_n$, $\Lambda^b\bV_{n-1}\hookrightarrow \Pi \Lambda^a\bV_n$, and respectively, $\Pi S^b\bV_{n-1}\hookrightarrow S^a\bV_n$, $\Pi \Lambda^b\bV_{n-1}\hookrightarrow \Lambda^a\bV_n$. Similar statements hold for the $\g(n)$-modules $S^a \bV_{n}^*$ and $\Lambda^a \bV_{n}^*$. Notice that the inequality $x_{n-1}<x_n$ holds whenever $m=\infty$.

Let $A = (a_1\leq a_2\leq \cdots)$ be a sequence of positive integers, and $\cA$ be a sequence of ordered pairs $(a_n,b_n)$, where $b_n\in \{0,1\}$ and $b_n = b_{n+1}$ if $a_n = a_{n+1}$. Then we define the $\g$-modules
\begin{eqnarray*}
S_\cA^\infty \bV & :=\varinjlim \Pi^{b_n} S^{a_n}\bV_n, \quad S_\cA^\infty \bV_* & :=\varinjlim \Pi^{b_n} S^{a_n}\bV_n^* \\
\Lambda_\cA^{\infty} \bV & :=\varinjlim \Pi^{b_n} \Lambda^{a_n}\bV_n, \quad \Lambda_\cA^{\infty} \bV_* & :=\varinjlim \Pi^{b_n} \Lambda^{a_n}\bV_n^*,
\end{eqnarray*}
where $\Pi^0$ is the identity functor. For $m=\infty$ this definition makes sense for any sequence $\cA$ as above, but for $m<\infty$ the $\g$-modules $\Lambda_\cA^{\infty} \bV$ and $\Lambda_\cA^{\infty} \bV_*$ are well defined only under the additional assumption that $a_{n+1}\in \{a_n, a_n+1\}$ and $b_n$ is constant for all $n\geq m+1$.

\subsection*{The modules $S^\mu \bV$ and $S^\mu \bV_*$} 
Let $\mu : = (\mu_1\geq \cdots \geq \mu_k)$ be a partition, and for every $n\geq k$ consider the weight $\lambda(n):=(\mu_1, \ldots, \mu_k, 0^{(n-k)}|0^{(x_n)})\in \fh(n)^*$. There are unique (up to scalar) embeddings of $\g(n)_0$-modules $L_{\fb(<_n)_0}(\lambda(n))\hookrightarrow L_{\fb(<_{n+1})_0}(\lambda(n+1))$ sending a $\fb(<_n)_0$-highest weight vector to a $\fb(<_{n+1})_0$-highest weight vector. Thus Proposition~\ref{prop:g_0-emb.g-emb} below implies that there are unique (up to scalar) embeddings of $\g(n)$-modules $\bL_{\fb(<_n)}(\lambda(n))\hookrightarrow \bL_{\fb(<_{n+1})}(\lambda(n+1))$ sending a $\fb(<_n)$-highest weight vector to a $\fb(<_{n+1})$-highest weight vector. Similar statements hold for the $\g(n)$-modules $\bL_{\fb(>_n)}(\lambda(n))^*$. Finally, we define the $\g$-modules
	\[
S^\mu \bV\cong \varinjlim \bL_{\fb(<_n)}(\lambda(n)), \quad S^\mu \bV_*\cong \varinjlim \bL_{\fb(>_n)}(\lambda(n))^*.
	\]

For all $n$, let $\lambda(n)\in \fh(n)^*$ be a weight of the following form:
\begin{enumerate}
\item[($\Omega_1$)] \phantomsection\label{sym.type.mod} $(a_n,0^{(n-1)}|0^{(x_n)})$,
\item [($\Omega_2$)] \phantomsection\label{sym.type.mod.dual} $(-a_n, 0^{(n-1)}|0^{(x_n)})$,
\item [($\Omega_3$)] \phantomsection\label{ext.type.mod} $(0^{(n)}|0^{(n-1)}, a_n)$,
\item [($\Omega_4$)] \phantomsection\label{ext.type.mod.dual} $(0^{(n)}|0^{(x_n-1)}, -a_n)$,
\item [($\Omega_5$)] \phantomsection\label{partition.type.mod} $(\mu_1,\ldots,\mu_k,0^{(n-k)}|0^{(x_n)})$,
\item [($\Omega_6$)] \phantomsection\label{partition.type.mod.dual} $(-\mu_1,\ldots,-\mu_k,0^{(n-k)}|0^{(x_n)})$,
\end{enumerate}
where $A=(a_1\leq a_2\leq \ldots)$ will be clear from the context, and $\mu : = (\mu_1\geq \cdots \geq \mu_k)$ is a partition.  Notice that
\begin{enumerate}
\item [($\Omega_1'$)] \phantomsection\label{sym.type.mod} $S_\cA^{\infty}\bV  = \varinjlim  \Pi^{b_n} S^{a_n}\bV_{n} \cong \varinjlim \Pi^{b_n} \bL_{\fb(<_n)}(\lambda(n))$,
\item [($\Omega_2'$)] \phantomsection\label{sym.type.mod.dual} $S_\cA^{\infty}\bV_*  =   \varinjlim  \Pi^{b_n} S^{a_n}\bV_{n}^* \cong \varinjlim \Pi^{b_n} \bL_{\fb(>_n)}(\lambda(n))$,
\item [($\Omega_3'$)] \phantomsection\label{ext.type.mod} $\Lambda_\cA^{\infty}\bV =  \varinjlim  \Pi^{b_n} \Lambda^{b_n}\bV_{n} \cong \varinjlim \Pi^{b_n}\bL_{\fb(>_n)}(\lambda(n))$,
\item [($\Omega_4'$)] \phantomsection\label{ext.type.mod.dual} $\Lambda_\cA^{\infty}\bV_* = \varinjlim  \Pi^{b_n} \Lambda^{b_n}\bV_{n}^* \cong \varinjlim \Pi^{b_n} \bL_{\fb(<_n)}(\lambda(n))$,
\item [($\Omega_5'$)] \phantomsection\label{partition.type.mod} $S^\mu \bV\cong \varinjlim (S^\mu \bV_n:= \bL_{\fb(<_n)}(\lambda(n)))$, $\Pi S^\mu \bV\cong \varinjlim (\Pi S^\mu \bV_n:= \Pi\bL_{\fb(<_n)}(\lambda(n)))$,
\item [($\Omega_6'$)] \phantomsection\label{partition.type.mod.dual} $S^\mu \bV_* \cong \varinjlim ( S^\mu \bV_n^* := \bL_{\fb(>_n)}(\lambda(n)))$, $\Pi S^\mu \bV_* \cong \varinjlim (\Pi S^\mu \bV_n^* := \Pi \bL_{\fb(>_n)}(\lambda(n)))$.
\end{enumerate}

\subsection*{Extensions}\label{sec:extensions}
For $n,m\in \Z_{>0}$, we set 
	\[
\rho(n|m) := (n,\ldots, 2,1 | -1,-2,\ldots, -m),
	\]
and, for any given weight $\lambda = (a_1,\ldots, a_n | b_1,\ldots, b_m)$ of $\fsl(m|n)$, we define the \emph{left side} (respectively, \emph{right side}) of $\lambda$ to be $(a_1,\ldots, a_n)$ (respectively, $(b_1,\ldots, b_m)$). 

Let $F$ be the set of all functions from $\Z$ to the set of symbols $\{<,>,\times, \circ\}$ such that $f(z) = \circ$ for all but finitely many $z\in \Z$. Define 
	\[
\#f := |f^{-1}(\times )|,\quad core_L(f) := f^{-1}(>),\quad core_R (f) := f^{-1}(<),
	\]
and let the \emph{core of} $f$ be
	\[
core(f) := (core_L(f), core_R(f)).
	\]
If $f\in F$, we define the \emph{weight diagram} $D_{wt}(f)$ to be the graph of the function $f$, i.e. a number line with the symbol $f(z)$ drawn at each $z\in \Z$. Also, if $\#f = k$, then we set $\times (f):= (a_1,\ldots, a_k)$, where $f^{-1}(\times) = \{a_1,\ldots, a_k\}$, and $a_1>\cdots>a_k$. If $a,b\in \Z$ satisfy $f(a)=\times$, $f(b) = \circ$ and $b<a$, we define $f_b^a\in F$ to be the map with same core as $f$, and such that
	\[
\times (f_b^a) = (a_1,\ldots,a_{j-1}, b, a_{j+1},\ldots, a_k),
	\]
where $a = a_j$ and $a_{j-1}<b<a_{j+1}$. Let $l_f(b,a)$ denote the number of occurrences of the symbol $\times$ minus the number of occurrences of the symbol $\circ$ strictly between $b$ and $a$ in $D_{wt}(f)$. We say that $g$ is obtained from $f$ by a \emph{legal move of weight zero} if $g = f_b^a$ for some $a,b\in \Z$ with $l_f(b,a)=0$.

Let $P\subseteq \Z^n\times \Z^m$ (respectively, $P^+\subseteq \Z^n\times \Z^m$) denote the set of integral (respectively, dominant integral) weights of $\fgl(m|n)$. Any $(\lambda_1,\ldots, \lambda_m | \lambda_1',\ldots, \lambda_n')\in P$ can be identified with the following $\rho(m|n)$-shifted element
	\[
(a_1:=\lambda_1+n,\ldots, a_n:=\lambda_n+1 | b_1:= 1 - \lambda_1',\ldots, b_m:= m - \lambda_m').
	\]
Via this identification, $P^+$ corresponds to the set of elements $\lambda = (a_1,\ldots, a_n | b_1,\ldots, b_m)\in P$ such that
	\[
a_1>\cdots>a_n,\quad b_1<\cdots<b_m.
	\]

For any $f\in F$, write
	\[
core_L(f) \cup \times (f) = (a_1>\cdots >a_n)\quad \text{and} \quad core_R(f) \cup \times (f) = (b_1<\cdots <b_m),
	\]
and set
	\[
\lambda_f := (a_1,\ldots, a_n | b_1,\ldots, b_m)\in P^+.
	\]
The map $F\ni f \mapsto \lambda_f\in P^+$ is a bijection between $F$ and $P^+$, and its inverse is $P^+\ni \lambda \mapsto f_\lambda\in F$.

Given $f,g\in F$, we write
	\[
f\to g,\quad g\to f
	\]
if $g$ is obtained from $f$ by a legal move of weight zero, or $f$ is obtained from $g$ by a legal move of weight zero, respectively. Let $\bL_{\fgl(n|m)}(\nu)$ denote a simple highest weight $\fgl(n|m)$-module of highest weight $\nu$ with respect to the Borel subalgebra of $\fgl(n|m)$ given by upper triangular matrices. Let $\fh(m|n)$ be the diagonal subalgebra of $\fgl(m|n)$. Then it follows from \cite[Theorem~B]{MS11} that $\Ext_{\fgl(n|m), \fh(m|n)}^1 (\bL_{\fgl(n|m)}(\lambda_f), \bL_{\fgl(n|m)}(\lambda_g))\neq 0$ if and only if $f\to g$ or $g\to f$, where the subscripts on $\Ext^1$ indicate that we consider extensions in the category of weight modules.

\begin{rem}\label{rem:reduction.to.gl(n|m)}
If $n\ne m$ then we have a direct sum of ideals $\fgl(n | m) =  \C z\oplus \fsl(n|m)$, where the identity matrix $z = I_{n+m}$ is central in $\fgl(n|m)$. Let $M$ be a simple object in the category of weight modules over $\fgl(n|m)$. Since $z$ lies in the center of $\fgl(n|m)$, we have an isomorphism of $\fgl(n|m)$-modules $M\cong \C_c\boxtimes S$, where $S = M|_{\fsl(n|m)}$ is a simple weight $\fsl(n|m)$-module and $\C_c$ is one-dimensional with $z$ acting on $\C_c$ via multiplication by $c$. Let $\C_c\boxtimes S$ and $\C_d\boxtimes T$ be two simple weight $\fgl(n|m)$-modules. Then
\begin{align*}
\Ext_{\C z\oplus \fsl(n|m)}^1(\C_c\boxtimes S, \C_d\boxtimes T)& \cong \Ext_{\C z}^1(\C_c, \C_d)\otimes \Hom_{\fsl(n|m)}(S,T) \\
& \oplus \Hom_{\C z}(\C_c, \C_d)\otimes \Ext_{\fsl(n|m)}^1(S,T),
\end{align*}
where in this remark we skip the Cartan subalgebras in the subscripts. Thus, if we assume that $S\ncong T$ and $c=d$, we obtain 
	\[
\Ext_{\C z\oplus \fsl(n|m)}^1(\C_c\boxtimes S, \C_c\boxtimes T) \cong  \Ext_{\fsl(n|m)}^1(S,T).
	\]

Let $\bL_{\fb(<_n)}(\lambda|\lambda')$ be a simple highest weight $\fsl(n|m)$-module and consider $c(\lambda) := \sum \lambda_i + \sum \lambda_i'$. In what follows we denote the $\fgl(n|m)$-module $\C_{c(\lambda)}\boxtimes \bL_{\fb(<_n)}(\lambda|\lambda')$ by $\bL_{\fgl}(\lambda|\lambda')$. Notice that for any other weight $(\nu|\nu')$ there exists $d(\lambda)\in \C$ such that $\C_{c(\lambda)}\boxtimes \bL_{\fb(<_n)}(\nu|\nu')\cong \bL_{\fgl}(\nu + d(\lambda)^{(n)}| \nu'-d(\lambda)^{(m)})$. Then
	\[
\Ext_{\fsl(n|m)}^1 (\bL_{\fb(<_n)}(\lambda|\lambda'), \bL_{\fb(<_n)}(\nu|\nu')) \cong \Ext_{\fgl(n|m)}^1(\bL_{\fgl}(\lambda|\lambda'), \bL_{\fgl}(\nu + d(\lambda)^{(n)}| \nu'-d(\lambda)^{(m)}))
	\]
\demo
\end{rem}

For the next result we need to write the $\g(n)$-modules appearing in {\rm \hyperref[sym.type.mod]{($\Omega_1'$)}-\hyperref[partition.type.mod.dual]{($\Omega_6'$)}} as $\fb(<_n)$-highest weight modules. The following isomorphisms of $\g(n)$-modules can be obtained via odd reflections (see \cite[Lemma~10.2]{Ser11}, or \cite[Lemma~3]{PS94}):
\begin{enumerate}
\item $S^{a_n}\bV_{n}^* = \bL_{\fb(>_n)}(-a_n, 0^{(n-1)}|0^{(x_n)}) \cong  \bL_{\fb(<_n)}(\lambda(n))$,
\item  $\Lambda^{a_n}\bV_{n} = \bL_{\fb(>_n)}(0^{(n)}|0^{(x_n-1)}, a_n) \cong \bL_{\fb(<_n)}(\lambda(n))$,
\item  $S^\mu \bV_n^* =  \bL_{\fb(>_n)}(-\mu_1,\ldots,-\mu_k,0^{(n-k)}|0^{(x_n)}) \cong \bL_{\fb(<_n)}(\lambda(n))$,
\end{enumerate}
where for $n > k$, the respective $\lambda(n)$ is as follows:
\begin{enumerate}
\item [(${\widetilde \Omega}_2$)] \phantomsection\label{sym.type.mod.dual.b(<)} $(0^{(n)}|0^{(x_n-a_n)}, -1^{(a_n)})$ if $a_n\leq x_n$, or $(0^{(n-1)}, -a_n + x_n| (-1)^{(x_n)})$ otherwise,
\item [(${\widetilde {\Omega}}_3$)] \phantomsection\label{ext.type.mod.b(<)}  $(1^{(a_n)},0^{(n-a_n)}|0^{(x_n)})$ if $a_n\leq n$, or $(1^{(n)}|a_n-n, 0^{(x_n-1)})$ otherwise,
\item [(${\widetilde{\Omega}}_6$)] \phantomsection\label{partition.type.mod.dual.b(<)} $(0^{(n-l)}, -\mu_l+x_n,\ldots,-\mu_1+x_n|-l^{(x_n - \mu_{l+1})},\ldots,-k^{(\mu_k)})$ if $\mu_l\geq x_n$ and $\mu_{l+1}<x_n$ for some $l$, or $(0^{(n)}|0^{(x_n-\mu_1)},-i^{(\mu_1 - \mu_{i+1})},\ldots,-k^{(\mu_k)})$ otherwise (in the latter case $i$ is such that $\mu_1=\cdots=\mu_i$ and $\mu_i>\mu_{i+1}$). In fact, both types of weights can be described by partitions: in the former case, to any pair of partitions $\nu = (\nu_1'\geq\cdots \geq \nu_{x_n}')$ and $\nu = (\nu_1\geq\cdots \geq \nu_p)$ we associate the weight $(0^{(n-p)},-\nu_p, \ldots, -\nu_1 |-\nu_{x_n}',\ldots, -\nu_1')$; in the latter case, to any partition $\nu = (\nu_1\geq\cdots \geq \nu_p)$ with $p\leq x_n$ we associate the weight $(0^{(n)} | 0^{(x_n-p)},-\nu_p,\ldots, -\nu_1)$.
\end{enumerate}

In the proof of the next result we use the symbol ``$\star$'' for a mark of a weight whose explicit form does not matter.

\begin{lem}\label{lem:not.hw.trivial.ext}
Assume that $x_n>1$ in $\g(n) = \fsl(n | x_n)$, and let $P, Q$ be simple $\g(n)$-modules occurring in {\rm \hyperref[sym.type.mod]{($\Omega_1'$)}-\hyperref[partition.type.mod.dual]{($\Omega_6'$)}}. Assume in addition that, if $P$ or $Q$ has type {\rm \hyperref[partition.type.mod]{($\Omega_5'$)}} or {\rm  \hyperref[partition.type.mod.dual]{($\Omega_6'$)}} then the length of the respective partition $\mu$ is much smaller than $n$. Then $\Ext_{\g(n), \fh(n)}^1(P,Q)=0$.
\end{lem}
\begin{proof}
Let $\lambda$ be the $\fb(<_n)$-highest weight of a module appearing in {\rm \hyperref[sym.type.mod]{($\Omega_1'$)}-\hyperref[partition.type.mod.dual]{($\Omega_6'$)}}, and set $f:=f_{\lambda}$. We claim that if $a<b$ satisfy $f(a) = \times$, $f(b) = \circ$ and $c\in \C$, then for $n\gg 0$ the weight $\lambda_{f_b^a}+(c^{(n)}|-c^{(x_n)})$ does not occur as a $\fb(<_n)$-highest weight of a module  in {\rm \hyperref[sym.type.mod]{($\Omega_1'$)}-\hyperref[partition.type.mod.dual]{($\Omega_6'$)}}. Below we prove this claim for $\lambda$ of the form $(a_n, 0^{(n-1)} | 0^{(x_n)})$ or $(0^{(n)} | 0^{(x_n-1)}, -a_n)$ for $a_n \in \Z_{>0}$, or $(\mu_1,\ldots, \mu_k, 0^{(n-k)} | 0^{(x_n)})$ for a partition $\mu = (\mu_1\geq \cdots \geq \mu_k)$. The other cases follow by dualization.

Performing an arbitrary legal move of weight zero on $f$ yields a weight whose $\rho(n|x_n)$-shifted form is given by
	\[
\lambda_{f^a_b} = (\star, \ldots, \star,b', b | b, c, \star, \ldots, \star),
	\]
where $|b' - b|>1$ and $|c-b|>1$. Since we are assuming $x_n\geq 2$, we conclude that $\lambda_{f^a_b}$ is not equal the $\rho(n|x_n)$-shifted form of the following weights: $(b_n, 0^{(n-1)} | 0^{(x_n)})$, $(0^{(n)} | 0^{(x_n-1)}, -b_n)$, $(0^{(n)}|0^{(x_n-b_n)}, -1^{(b_n)})$, $(0^{(n-1)}, -b_n + x_n| (-1)^{(x_n)})$, $(1^{(b_n)},0^{(n-b_n)}|0^{(x_n)})$ or $(1^{(n)}|b_n-n, 0^{(x_n-1)})$ for $b_n\in \Z_{>0}$, $(\nu_1,\ldots, \nu_l, 0^{(n-l)} | 0^{(x_n)})$, $(0^{(n)}|0^{(x_n-\nu_1)},-i^{(\nu_1 - \nu_{i+1})},\ldots,-l^{(\nu_l)})$ for a partition $\nu = (\nu_1\geq \cdots \geq \nu_l)$.

To prove that the weight $\lambda_{f^a_b}$ is not equal the $\rho(n|x_n)$-shifted form of a weight $(0^{(n-l)}, -\nu_l+x_n,\ldots,-\nu_1+x_n|-l^{(x_n - \nu_{l+1})},\ldots,-k^{(\nu_l)})$ for a partition $\nu = (\nu_1\geq \cdots \geq \nu_l)$, we notice that if $\lambda$ equals $(a_n, 0^{(n-1)} | 0^{(x_n)})$ (respectively, $(\mu_1,\ldots, \mu_k, 0^{(n-k)} | 0^{(x_n)})$), the difference of the $n$-th and $(n-1)$-th (respectively, the $k$-th and $(k+1)$-th) marks in the left side of $\lambda_{f_b^a}$ is bigger than zero (here we are assuming that $n\gg 0$ so that $n-l>k$). For $\lambda=(0^{(n)} | 0^{(x_n-1)}, -a_n)$ we take $\min \{n, a_n\} \gg 0$ so that $x_n+a_n\gg l$, and: the difference of the $x_n$-th and $(x_n-1)$-th marks in the right side of $\lambda_{f_a^b}$ is bigger than $k$ (if $a = x_n+a_n$), or the difference of the $(x_n+a_n-2)$-th and $(x_n+a_n-1)$-th marks in the left side of $\lambda_{f_b^a}$ is bigger than $1$ (if $a < x_n+a_n$). This proves the claim.

Let $\nu$ be the $\fb(<_n)$-highest weight of a module occurring in {\rm \hyperref[sym.type.mod]{($\Omega_1'$)}-\hyperref[partition.type.mod.dual]{($\Omega_6'$)}}. We have shown in all cases that there exists a pair of marks of $\lambda_{f_b^a}$ whose difference does not coincide with the difference of the respective pair of marks of the $\rho(n|x_n)$-shifted form of $\nu$. Since for any $c\in \C$ the difference of any pair of marks of $\lambda_{f_b^a} + (c^{(n)}|-c^{(x_n)})$ coincides with the difference of the respective pair of marks of $\lambda_{f_b^a}$, we conclude (1): for any $c\in 
\C$ the non-shifted form of $\lambda_{f_b^a} + (c^{(n)}|-c^{(x_n)})$ cannot occur as a $\fb(<_n)$-highest weight of a module in {\rm \hyperref[sym.type.mod]{($\Omega_1'$)}-\hyperref[partition.type.mod.dual]{($\Omega_6'$)}}.

Assume now $\mu$ is one of the $\fb(<_n)$-highest weights appearing in {\rm \hyperref[sym.type.mod.dual.b(<)]{(${\widetilde \Omega}_2$)}, \hyperref[ext.type.mod.b(<)]{(${\widetilde \Omega}_3$)}, \hyperref[partition.type.mod.dual.b(<)]{(${\widetilde \Omega}_6$)}} and set $g=f_\mu$.  Similarly to (1) we show (2): if $g^a_b$ is obtained from $g$ by a legal move of weight zero, then for any $c\in 
\C$ the non-shifted form of $\lambda_{g^a_b} + (c^{(n)}|-c^{(x_n)})$ does not occur as a $\fb(<_n)$-highest weight of a module in {\rm \hyperref[sym.type.mod]{($\Omega_1'$)}-\hyperref[partition.type.mod.dual]{($\Omega_6'$)}}. Now we can combine (1) and (2) above with \cite{MS11} to obtain $\Ext_{\fgl(n|x_n)\oplus \C z}^1(\bL_{\fgl}(\nu), \bL_{\fgl}(\lambda + (c^{(n)}|-c^{(x_n)}))=0$ for every $c\in \C$ and any weight $\nu$ occurring as a $\fb(<_n)$-highest weight of a module in {\rm \hyperref[sym.type.mod]{($\Omega_1'$)}-\hyperref[partition.type.mod.dual]{($\Omega_6'$)}}. Finally, Remark~\ref{rem:reduction.to.gl(n|m)} gives
	\[
\Ext_{\fsl(n|x_n), \fh(n)}^1 (\bL_{\fb(<_n)}(\nu), \bL_{\fb(<_n)}(\lambda)) \cong \Ext_{\fgl(n|x_n), \fh(n)\oplus \C z}^1(\bL_{\fgl}(\nu), \bL_{\fgl}(\lambda + (c(\nu)^{(n)}|-c(\nu)^{(x_n)})) = 0,
	\]
and the statement follows.
\end{proof}

\subsection{Main results}

Recall the $\fsl(\infty)$-modules $\Lambda_A^{\frac{\infty}{2}}V$, $S_A^\infty V$, $S_A^\infty V_*$, $S^\mu V$ and $S^\mu V_*$ defined in Section~\ref{sec:int.bound.gl}. The support of each of these modules equals the projection to $\fh_\fsl^*$ of a respective subset of $\C^\infty$:
\begin{enumerate} \label{list:types.of.weights}
\item [(i)] \phantomsection\label{list:types.of.weights.exterior}  $\Lambda_A := \{\varepsilon_B=\sum_{i\in B} \varepsilon_i \mid B\approx A\}$, where $B \approx A$ means that there exist disjoint finite subsets $F_A\subseteq A$ and $F_B\subseteq B$, such that $|F_A|=|F_B|$ and $A\setminus F_A = B\setminus F_B$,
\item [(ii)] \phantomsection\label{list:types.of.weights.symmetric}  $S_A := \{\lambda \mid \lambda_i\geq 0, \exists n : \sum_{i=1}^n \lambda_i = a_n, \lambda_i = (a_i - a_{i-1}) \text{ for } i>n\}$, where $a_i\in A$,
\item [(iii)] \phantomsection\label{list:types.of.weights.symmetric.dual}  $S_A^* := \{\lambda \mid \lambda_i\leq 0, \exists n : \sum_{i=1}^n \lambda_i = -a_n, \lambda_i = (a_{i-1} - a_{i}) \text{ for } i>n\}$, where $a_i\in A$,
\item [(iv)] \phantomsection\label{list:types.of.weights.schur}  $S_\mu := \{\lambda \mid 0\leq \lambda_i\leq \mu_i\}$,
\item [(v)] \phantomsection\label{list:types.of.weights.schur.dual}  $S_\mu^* := \{\lambda \mid 0\leq -\lambda_i\leq \mu_i\}$.
\end{enumerate}

Let $m\in \Z_{\geq 1}\cup\{\infty\}$. In this section $M$ is assumed to be a simple integrable bounded $\fsl(\infty|m)$-module. We use the symbol ``$\diamond$'' for a weight whose explicit form does not matter.

\begin{lem}\label{lem:reduction.to.sl.weights}
Any weight of $M$ can be obtained as the projection of a vector $(\nu | \diamond)$, where $\nu$ lies in one of the subsets displayed in {\rm \hyperref[list:types.of.weights.exterior]{(i)}-\hyperref[list:types.of.weights.schur.dual]{(v)}}.
\end{lem}
\begin{proof}
Any weight of $M$ is a weight of some $M(i)$, and hence, as discussed in the beginning of Section~\ref{sec:classification.type.A}, it can be obtained as the projection of some vector $(\nu^c | \diamond)\in \C^\infty\times \C^m$, where $c\in \C$ and $\nu$ lies in one of the subsets displayed in \hyperref[list:types.of.weights.exterior]{(i)}-\hyperref[list:types.of.weights.schur.dual]{(v)}. Since the projection of $( \nu^c | \diamond)$ to $\fh^*$ coincides with the projection of $(\nu^c - c^{(\infty)} + \nu_1^{(\infty)} | \diamond + c^{(m)} - \nu_1^{(m)}) = (\nu | \diamond + c^{(m)} - \nu_1^{(m)})$, the statement follows.
\end{proof}

Let $v\in M(i)\subseteq M$ be a nonzero weight vector with $M(i)|_{\g_0'}$ isomorphic to $S\boxtimes T$, where $S$ (respectively, $T$) is an integrable bounded simple weight $\fsl(\infty)$-module (respectively, $\fsl(m)$-module). If $S$ is isomorphic to $\Lambda_A^{\frac{\infty}{2}}V$, $S_A^\infty V$, $S_A^\infty V_*$, $S^\mu V$, or $S^\mu V_*$, then we say that $v$ has type \hyperref[list:types.of.weights.exterior]{(i)}, \hyperref[list:types.of.weights.symmetric]{(ii)}, \hyperref[list:types.of.weights.symmetric]{(iii)}, \hyperref[list:types.of.weights.schur]{(iv)}, or \hyperref[list:types.of.weights.schur.dual]{(v)}, respectively.

\begin{lem}\label{lem:components.same.type}
Let $v\in M^{(\nu|\diamond)}$ be a nonzero weight vector with type  $(*) \in \{\rm \hyperref[list:types.of.weights.exterior]{(i)} \rm - \hyperref[list:types.of.weights.schur.dual]{(v)}\}$. If $w\in M$ is a nonzero weight vector, then $w$ also has type $(*)$.
\end{lem}
\begin{proof}
Since $M$ is simple, it is enough to prove that the action of $\g_{\bar 1}$ on $v$ does not change the type of $v$. Assume that $v\in M(i)\cong S\boxtimes T$, where $S$ (respectively, $T$) is an integrable bounded simple weight $\fsl(\infty)$-module (respectively, $\fsl(m)$-module). Let $w := X_{\alpha}v$, where $X_{\alpha} \in \g_\alpha\subseteq \g_{\bar 1}$. Take $n\gg 0$ so that the root vectors $X_{\pm (\delta_i - \delta_j)}$ commute with $X_\alpha$ for all $i,j\geq n$. Let $\fs$ denote the Lie subalgebra of $\g_0$ generated by all such root vectors. Notice that $\fs \cong \fsl(\infty)$, and $\bU(\fs)w = X_\alpha \bU(\fs)v$. Thus we have an isomorphism of $\fs$-modules $\bU(\fs)w\cong \bU(\fs)v$, and using the fact that $S$ is isomorphic to one of the modules listed in the beginning of this section, we easily check that the type of $\bU(\fs)w$ coincides with the type of $S$. Precisely, if $S$ is isomorphic to $S_A^\infty V$ or $S_A^\infty V_*$ for an infinite set $A\subseteq \Z_{>0}$, to $ \Lambda_A^{\frac{\infty}{2}} V$ for a semi-infinite set $A\subseteq \Z_{>0}$, or to $S^\mu V$, $S^\mu V_*$ for a partition $\mu = (\mu_1\geq \cdots \geq \mu_k)$, then $\bU(\fs)w$ is isomorphic respectively to $S_B^\infty V$, $S_B^\infty V_*$, $\Lambda_B^{\frac{\infty}{2}} V$, $S^\eta V$ or $S^\eta V_*$, where $B = \{b_1\leq b_2\leq \cdots\}\subseteq \Z_{\geq n}$ satisfies $b_i = a_{n+i}$ for all $i\geq n$, and $\eta = (\eta_1\geq \cdots \geq \eta_l)$ is the partition determined by the weight $\mu|_{\fh\cap \fs}\in (\fh\cap \fs)^*$. Therefore, the assumption that $v$ and $w$ have different types would contradict to the fact that both $\fs \cong \fsl(\infty)$-modules $\bU(\fs)v$ and $\bU(\fs)w$ have the type of $S$.
\details{ 
Let's consider now the case where the components $S$ of all $M(i)$'s is isomorphic to $S^\mu V$ or $S^\mu V_*$ for a partition $\mu = (\mu_1\geq \cdots \geq \mu_k)$. Notice that we may assume that $\mu$ is nonempty for some $i$. Indeed, if $M(i)\cong S^\emptyset V$ for all $i$, then $\g_{\bar 1}M(i)=0$, and hence $\g M(i)=0$.  This implies $\g M=0$. Then we may assume that $S\cong S^\mu V$ (respectively, $S^\mu V_*$) for a nonempty partition $\mu = (\mu_1\geq \cdots \geq \mu_k)$. Consider the Borel subalgebra $\fb(>)$ of $\g$. Without loss of generality we can assume that $v$ is $(\fb(>)_{\bar 0}\cap \fsl(\infty))$-singular, and hence that its weight is of the form $(0^{(\infty)}, \mu_k,\cdots, \mu_1 | \diamond)$. It is clear that all elements $X_{\delta_i-\varepsilon_j}$ preserve the type of $v$. So it remains to show that $\fb(>)_{\bar 1}$ preserve the type of $v$. For this, observe first that $X_{\varepsilon_l-\delta_j}v=0$ for all $1\leq l\leq m$ and $j>k$ (by support arguments). Now, let $u_k\in \bU(\fb(>)_{\bar 1})$ be the maximal monomial in the variables $X_{\varepsilon_l-\delta_k}$ ($k$ is fixed) such that $u_kv\neq 0$ (the set of nonzero vectors of the form $X_{\varepsilon_{i_r}-\delta_k}\cdots X_{\varepsilon_{i_1}-\delta_k}v$ must be finite: indeed, all these vectors are $(\fb(>)\cap\fsl(\infty))$-singular and the existence of infinitely many such vectors would imply the existence of a $(\fb(>)\cap\fsl(\infty))$-singular vector of weight $(0^{(\infty)}, -1,\mu_{k-1},\cdots, \mu_1 | \diamond)$, which contradicts the integrability of $M$ with respect to the $\fsl(2)$-triple $\{X_{\delta_{k+1}-\delta_{k}}, X_{\delta_{k}-\delta_{k+1}}, H_{\delta_{k}-\delta_{k+1}}:=[X_{\delta_{k+1}-\delta_{k}}, X_{\delta_{k}-\delta_{k+1}}]\}$). Since $u_kv$ is $(\fb(>)_{\bar 0}\cap \fsl(\infty))$-singular, its weight has to be of the form $(0^{(\infty)}, \mu_k',\mu_{k-1},\cdots, \mu_1 | \diamond)$ with $\mu_k'\geq 0$. Also, by the construction of $u_k$, we have $X_{\varepsilon_l-\delta_j}u_kv=0$ for all $1\leq l\leq m$ and $j>k-1$. Now, repeating this argument inductively (replacing $v$ by $u_kv$ and $k$ by $k-1$) we construct a vector $w=u_1\cdots u_{k-1}u_kv$ of weight $(0^{(\infty)}, \mu_k',\cdots, \mu_1' | \diamond)$ with $0\leq \mu_k'\leq \ldots \leq \mu_1'$ which is $\fb(>)_{\bar 1}$-singular. Now, to understand the action $\g_{\bar 1}$ on $w$ we just have to look at how the vectors $X_{\delta_i-\varepsilon_j}$ act on it. But these clearly do not change type $w$. Since $M$ is generated by $w$ the result follows, since all weight vectors will have the same type as $w$, which in its turn has the same type as $v$.

The proof for $S\cong S^\mu V_*$ for a nonempty partition $\mu = (\mu_1\geq \cdots \geq \mu_k)$ is similar.
}
\end{proof}

If $v, w\in M$ are nonzero weight vectors then Lemma~\ref{lem:components.same.type} allows us to claim that $v$ and $w$ have the same type according to \hyperref[list:types.of.weights.exterior]{(i)}-\hyperref[list:types.of.weights.schur.dual]{(v)}. Moreover, it follows from Lemma~\ref{lem:reduction.to.sl.weights} that if $v$ has type $(*) \in {\rm \{\hyperref[list:types.of.weights.exterior]{(i)}-\hyperref[list:types.of.weights.schur.dual]{(v)}\}}$ then its weight can be represented by the vector $(\nu, \diamond)$, with $\nu$ lying in a set of type $(*)$. In what follows we often use this fact.

\begin{lem}\label{lem:possible.h.w}
Let $v\in M^{(\diamond | \diamond)}$ be a nonzero weight vector, and consider the finite-dimensional $\g(n)$-module $M_n := \bU(\g(n)) v$. Let $P$ be a simple subquotient of $M_n$ and let $(\lambda | \gamma)\in \Supp P$. Then the following statements hold for $n\gg 0$:
\begin{enumerate}
\item If $v$ is of type {\rm \hyperref[list:types.of.weights.schur.dual]{(iv)}}, then any $\fb(<_n)$-singular weight $(\lambda | \gamma)$ of $P$ is of the form \linebreak[4] $(\mu_1,\ldots,\mu_k,0^{(n-k)}|0^{(x_n)})$ for a partition $\mu = (\mu_1\geq \cdots \geq \mu_k)$, or $(0^{(\infty)}| 0^{(x_n)})$, or $(1^{(n)}|a, 0^{(x_n-1)})$ where $a\in \Z_{\geq 0}$ if $x_n>1$ and $a\in \C$ if $x_n=1$.
\item If $v$ is of type {\rm \hyperref[list:types.of.weights.schur]{(v)}}, then any $\fb(>_n)$-singular weight  $(\lambda | \gamma)$ of $P$ is of the form \linebreak[4] $(-\mu_1,\ldots,-\mu_k,0^{(n-k)}|0^{(x_n)})$ for a partition $\mu = (\mu_1\geq \cdots \geq \mu_k)$, or $(0^{(\infty)}| 0^{(x_n)})$, or $(-1^{(n)}|-a, 0^{(x_n-1)})$  where $a\in \Z_{\geq 0}$ if $x_n>1$ and $a\in \C$ if $x_n=1$.
\item  If $v$ is of type {\rm \hyperref[list:types.of.weights.symmetric]{(ii)}}, then any $\fb(<_n)$-singular weight $(\lambda | \gamma)$ of $P$ is of the form \linebreak[4] $(a,0^{(n-1)} | 0^{(x_n)})$ for some $a\in \Z_{>0}$, or $(0^{(\infty)}| 0^{(x_n)})$.
\item If $v$ is of type {\rm \hyperref[list:types.of.weights.symmetric]{(iii)}}, then any $\fb(>_n)$-singular weight $(\lambda | \gamma)$ of $P$ is of the form \linebreak[4] $(-a,0^{(n-1)} | 0^{(x_n)})$ for some $a\in \Z_{>0}$, or $(0^{(\infty)}| 0^{(x_n)})$.
\item  If $v$ is of type {\rm \hyperref[list:types.of.weights.exterior]{(i)}}, then any $\fb(<_n)$-singular weight $(\lambda | \gamma)$ of $P$ is of the form \linebreak[4] $(1^{(a)}, 0^{(n-a)}|0^{(x_n)})$, or $(1^{(n)}|a, 0^{(x_n-1)})$ for some $a\in \Z_{\geq 0}$, or $(0^{(\infty)}| 0^{(x_n)})$.
\end{enumerate}
Moreover, in all above cases $(\lambda | \gamma) = (0^{(\infty)}| 0^{(x_n)})$ implies $\g(n)P=0$.
\end{lem}
\begin{proof}
Write $(\lambda|\gamma) = (\lambda_n, \ldots, \lambda_1 | \gamma_1,\ldots, \gamma_{x_n})$ and let $w\in P$ be a nonzero vector of weight $(\lambda | \gamma)$. Since $M_n$ is a finite-dimensional (and hence a semisimple) weight module of $\g(n)_0$ we may assume that $P$ is a $\g(n)_0$-submodule of $M_n$, and therefore that $w$ is a $\fb(<_n)_0$-singular vector of $M_n$.

(a).\details{The idea here is that we start with $v$ which has weight of the form $(0^{(\infty)},\nu_1,\ldots, \nu_l,0^{(n-l)}|\diamond)$ for $n\gg0$, and then we start acting on it with $\fb(<_n)$. So, in this case (a), when we write a weight of the form $(\lambda_n, \ldots, \lambda_1 | \diamond)$ we actually mean $(0^{(\infty)},\lambda_n, \ldots, \lambda_1 | \diamond)$. Same in (b). For the other cases we mean $(\diamond^{(\infty)},\lambda_n, \ldots, \lambda_1 | \diamond)$, where $\diamond^{(\infty)}$ stands by the continuation (not detectable by $\fh_n$) of the weight of $v$} Since $(\lambda|\gamma)$ is a $\fb(<_n)_0$-singular weight and $w$ has the same type of $v$ by Lemma~\ref{lem:components.same.type}, we must have $\lambda = (\mu_1,\ldots,\mu_k,0^{(n-k)})$ for some partition $\mu = (\mu_1\geq \cdots \geq \mu_k)$ where $k=1,\ldots, n$. Assuming $k<n$, we will show that $\gamma = 0$. For any $1\leq \ell\leq x_n$ we have $X_{\varepsilon_\ell - \delta_1}w=0$ as otherwise $X_{\varepsilon_\ell - \delta_1}w$ would be a vector of weight $(\mu_1,\ldots,\mu_k,0^{(n-k-1)},-1|\diamond)$ in contradiction to Lemma~\ref{lem:components.same.type}. Indeed, a weight vector with such a weight cannot have the type of $v$. Thus $X_{\varepsilon_\ell - \delta_1}w=0$. Since $w$ is a $\fb(<_n)$-singular weight vector, we conclude that $h_{\varepsilon_\ell - \delta_1}w=\gamma_\ell w = 0$, which shows $\gamma_\ell=0$. Since $\ell$ was arbitrary, this proves that $\gamma=0$.

If $k=n$ for all $n\gg 0$, then we must have $\lambda_i=1$ for all $i=1,\ldots, n$ as otherwise, by \cite[Theorem~5.1]{GP20}, $M$ would not be a bounded $\g_0$-module. Thus $(\lambda | \gamma) = (1^{(n)}|\gamma)$, and the statement is proved for $x_n=1$. Assume now that $x_n>1$. We claim that $\gamma=(a, 0^{(x_n-1)})$ or $\gamma=(-1^{(x_n-1)}, -a)$ for some $a\in \Z_{\geq 0}$. Indeed, if $\gamma_1\notin \Z$ or $\gamma_1\in \Z_{\leq -2}$ then, as in the previous case, we get a contradiction due to Lemma~\ref{lem:components.same.type}, since $X_{\varepsilon_2 - \delta_1} X_{\varepsilon_1 - \delta_1}w$ would be a nonzero vector of weight $(1^{(n-1)}, -1|\diamond)$. If $\gamma_1 \in \Z_{\geq 0}$ then $\gamma_i = 0$ for all $i\geq 2$ by the same reason. Finally, if $\gamma_1 = -1$ we can use again  Lemma~\ref{lem:components.same.type} to show that $\gamma_i = -1$ for all $2\leq i\leq x_n-1$ and that $\gamma_{x_n}=-a$ for some $a\in \Z_{\geq 0}$. The claim is proved. 

Notice that there are isomorphisms of $\g(n)$-modules
\begin{align*}
& \bL_{\fb(<_n)}(1^{(n)}|a, 0^{(x_n-1)}) \cong \Lambda^{n+a}\bV_n, \\
& \bL_{\fb(<_n)}(1^{(n)} | -1^{(x_n-1)}, -a) = \bL_{\fb(<_n)}(0^{(n)} | 0^{(x_n-1)}, -a+1)\cong \Lambda^{a-1}\bV_n^*,
\end{align*}
and, by Lemma~\ref{lem:components.same.type}, the latter module cannot occur as a $\g(n)$-subfactor of $M$ since vectors of $\Lambda^{a-1}\bV_n^*$ cannot have the type of $v$. 

To prove that $(\lambda | \gamma) = (0^{(\infty)}| 0^{(x_n)})$ implies $\g(n)P=0$ in case (a), notice that $(\g(n)_0\oplus \g(n)_{ 1}) w=0$ since  $w$ is a $\fb(<_n)$-singular  vector of weight $(0^{(\infty)}| 0^{(x_n)})$. Furthermore, $\g(n)_{-1}w\neq 0$ contradicts Lemma~\ref{lem:components.same.type}. Therefore $\g(n)w=0$ for any $\fb(<_n)$-singular vector of $P$, and consequently $\g(n)P=0$.

The remaining claims are proven in a similar way.
\end{proof}

\begin{rem}\label{rem:not.hw.trivial.ext.n=1}
For $\g = \fsl(n | 1)$, we have a weaker version of Lemma~\ref{lem:not.hw.trivial.ext}: if $P$, $Q$ are finite-dimensional simple $\g(n)$-modules whose respective $\fb(<_n)$-highest weights $\lambda$, $\mu$ are as in Lemma~\ref{lem:possible.h.w} (a) (respectively, (b)-(e)), then $\Ext_{\g(n), \fh(n)}^1 (P, Q)=0$. To prove this, we proceed as in Lemma~\ref{lem:not.hw.trivial.ext}: we show that $f_\lambda$ cannot be obtained from $f_\mu$ by a legal move of weight zero and vice-versa, and then we apply \cite{MS11}. \demo
\end{rem}

\begin{cor}\label{cor:M_n.semisimple}
Let $v\in M$ be a nonzero weight vector, and consider the finite-dimensional $\g(n)$-module $M_n = \bU(\g(n)) v$. If $P, Q$ are simple subquotients of $M_n$, then $\Ext_{\g(n), \fh(n)}^1 (P,Q)=0$. In particular, $M_n$ is a semisimple $\g(n)$-module.
\end{cor}
\begin{proof}
The highest weights allowed for $P$ and $Q$ are the ones occurring in Lemma~\ref{lem:possible.h.w} (a) (respectively, (b)-(e)). The statement now follows from Lemma~\ref{lem:not.hw.trivial.ext} for $m>1$, and from Remark~\ref{rem:not.hw.trivial.ext.n=1} for $m=1$.
\end{proof}

\begin{lem}\label{lem:M_n.simple}
If $v\in M^{(\lambda | \gamma)}$ is a nonzero vector, then $M_n = \bU(\g(n)) v$ is a simple $\g(n)$-module for all $n\gg 0$.
\end{lem}
\begin{proof}
By Lemma~\ref{lem:A1} from the Appendix, there exists $N\gg 0$ such that $M^{(\lambda | \gamma)}$ is a simple $\bU^0_N$-module. A standard argument shows that $M_n$ is a simple $\g(n)$-module for all $n\geq N$. Indeed, by Corollary~\ref{cor:M_n.semisimple}, any submodule $K\subseteq M_n$ yields a split exact sequence of $\g(n)$-modules
	\[
0\to K\to M_n\to W\to 0.
	\]
This sequence provides an exact sequence of $\bU_n^0$-modules
	\[
0\to K^{(\lambda | \gamma)}\to M_n^{(\lambda | \gamma)}\to W^{(\lambda | \gamma)}\to 0.
	\]
Since $M_n^{(\lambda | \gamma)}$ is a simple $\bU_n^0$-module, we have $K^{(\lambda | \gamma)} = 0$ or $K^{(\lambda | \gamma)} = M_n^{(\lambda | \gamma)}$. If $K^{(\lambda | \gamma)} = 0$ then $v\in W$
 and $M_n = \bU(\g(n))v = W$, which implies $K=0$. Similarly, if $K^{(\lambda | \gamma)} = M_n^{(\lambda | \gamma)}$ we conclude that $M_n = K$.
\end{proof}

\begin{theo}\label{thm:classification.sl.case}
Let $\g = \fsl(\infty|m)$ for $m\in \Z_{\geq 1}\cup\{\infty\}$ and let $M$ be an integrable bounded simple weight $\g$-module. Then the following statements hold:
\begin{enumerate}
\item $M$ is locally simple.
\item $M$ is isomorphic to one of the following modules: $S^\mu \bV$, $S^\mu \bV_*$, $\Pi S^\mu \bV$, $\Pi S^\mu \bV_*$, $S_\cA^\infty \bV$, $S_\cA^\infty \bV_*$, $\Lambda_\cA^{\infty} \bV$, or $\Lambda_\cA^{\infty} \bV_*$. If $m=1$, then $M$ can also be isomorphic to $\bL_{\fb(>)}(0^{(\infty)}|a)$ or $\bL_{\fb(<)}(0^{(\infty)}|a)$ for $a\in \C\setminus \Z$.
\item All isomorphisms between simple modules appearing in (b) are: $S_\cA^\infty\bV\cong S_{\cA'}^\infty\bV$, $S_\cA^\infty\bV_*\cong S_{\cA'}^\infty\bV_*$, $\Lambda_\cA^\infty\bV\cong \Lambda_{\cA'}^\infty\bV$ and $\Lambda_\cA^\infty\bV_*\cong \Lambda_{\cA'}^\infty\bV_*$ if and only if there exists $N>0$ such that $(a_i, b_i)=(a_i', b_i')$ for all $i\geq N$; $S^\emptyset \bV\cong S^\emptyset \bV_*\cong \C$ and $\Pi S^\emptyset \bV\cong \Pi S^\emptyset \bV_*\cong \Pi \C$ ($\emptyset$ stands for the empty partition).
\end{enumerate}
\end{theo}
\begin{proof}
Let $v\in M^{(\lambda | \gamma)}\setminus \{0\}$. By Lemma~\ref{lem:M_n.simple} the $\g(n)$-module $M_n = \bU(\g(n)) v$ is simple for all $n\gg 0$. In particular, $M = \bigcup_n M_n$ and $M$ is locally simple. This proves part (a). Part (b) follows from Lemma~\ref{lem:possible.h.w}. \details{For the case $m=1$, we can have a sequence of $\fb(<_n)$-highest weights of the form $(0^{(\infty)},1^{(n)}|a_n)$ for $a_n\in \C$ (see Lemma~\ref{lem:possible.h.w}~(a)). Assume first that $a_n\notin \Z_{\geq -n}$ for some $n$. Then, by using odd reflections we see that $M_n\cong \bL_{\fb(<_n)}(0^{(\infty)},1^{(n)}|a_n)\cong \bL_{\fb(>_{n})}(0^{(\infty)},0^{(n)}|a_n+n)$ and $M_{n+1}\cong \bL_{\fb(<_{n+1})}(0^{(\infty)},1^{(n+1)}|a_n-1)\cong \bL_{\fb(>_{n+1})}(0^{(\infty)},0^{(n+1)}|a_n+n)$. Indeed, the highest weight of $M_n$ is $(0^{(\infty)},1^{(n)}|a_n)$ and the highest weight of $M_{n+1}$ is $(0^{(\infty)},1^{(n+1)}|a_{n+1})$, but since it has to be obtained from $(0^{(\infty)},1^{(n)}|a_n)$ by the action of $X_{\delta_{n+1}-\varepsilon}$ (otherwise, since $X_{\delta_{i}-\varepsilon}$ annihilates $(0^{(\infty)},1^{(n)}|a_n)$ for all $1\leq i\leq n$ the $\fb(<_{n+1})$-highest weight of $M_{n+1}$ would be $(0^{(\infty)},1^{(n)},0|a_n)$, contradicting case (a) of Lemma~\ref{lem:possible.h.w}) we see that $a_{n+1}=a_n-1$. Thus $M\cong \bL_{\fb(>)}(0^{(\infty)}|a)$ for $a=a_n+n\in \C\setminus \Z_{\geq 0}$. Moreover, if $a\in \Z_{<0}$, then $\bL_{\fb(>_{n})}(0^{(\infty)},0^{(n)}|a)=\bL_{\fb(>_{n})}(-1^{(\infty)},-1^{(n)}|a+1)\cong \Lambda^{n+a+1}\bV_n^*$ for all $n$, which implies $M\cong \Lambda_\cA^{\infty} \bV_*$ for a suitable $\cA$. If $a_k\in \Z_{\geq -k}$ for some $k$, then since $a_{k+1}=a_k-1$ for all $k$, we may assume that $n\gg 0$ so that $a_n\in \{-1,\ldots, -n\}$. Now, by using odd reflections we see that $M_n\cong \bL_{\fb(<_n)}(0^{(\infty)},1^{(n)}|a_n)\cong \bL_{\fb(>_{n})}(0^{(\infty)},0^{(-a_n-1)},1^{(n+a_n+1)}|-1)= \bL_{\fb(>_{n})}(-1^{(\infty)},-1^{(-a_n-1)},0^{(n+a_n+1)}|0)\cong \Lambda^{-a_n-1}\bV_n^*$, and hence we obtain $M_{n+1}\cong \bL_{\fb(>_{n+1})}(-1^{(\infty)},-1^{(-a_n)},0^{(n+a_n)}|0) \cong \Lambda^{-a_n}\bV_{n+1}^*$. This implies $M\cong \Lambda_\cA^{\infty} \bV_*$ for a suitable $\cA$. Similarly, if we have a sequence of $\fb(>_n)$-highest weights of the form $(-1^{(n)}|a_n)$ (see Lemma~\ref{lem:possible.h.w}~(b)), then we obtain that $M$ is isomorphic to a module of the form $\bL_{\fb(<)}(0^{(\infty)}|a)$ for some $a\in \C$, or to $\Lambda_\cA^{\infty} \bV$ for a suitable $\cA$}. Finally, one direction of (c) is clear, the other follows from the observation that if a locally simple module $M$ is isomorphic to $\varinjlim M_n$ and to $\varinjlim M_n'$, then $M_n\cong M_n'$ for $n\gg 0$. \details{For any given weight vector, its weight is determined by the sequence of embeddings defining the module, and, in the other direction, any sequence of embeddings can be obtained by looking at the weight any given weight vector. Since isomorphisms preserve weights the result follows.}
\end{proof}

Suppose that $\g=\fsl(\infty|m)$ with $m<\infty$, and that $M$ is isomorphic to $S_\cA^\infty \bV$. Notice that, for all $n\geq m+1$, if $M_n\cong S^{a_n} \bV_n$ (respectively, $M_n\cong \Pi S^{a_n} \bV_n$) then $M_{n+1}\cong S^{a_{n+1}} \bV_{n+1}$ (respectively, $M_{n+1}\cong \Pi S^{a_{n+1}} \bV_{n+1}$). For the case where $M$ is isomorphic to $S_\cA^\infty \bV_*$, $\Lambda_\cA^{\infty} \bV$ or $\Lambda_\cA^{\infty} \bV_*$ the situation is analogous. Thus Theorem~\ref{thm:classification.sl.case} can be refined as follows:

\begin{cor}\label{cor:classification.sl.case}
If $M$ is an integrable bounded simple weight $\fsl(\infty | m)$-module ($m<\infty$), then $M$ is isomorphic to one of the following modules: $S^\mu \bV$, $S^\mu \bV_*$, $\Pi S^\mu \bV$, $\Pi S^\mu \bV_*$, $S_\cA^\infty \bV$, $S_\cA^\infty \bV_*$, $\Lambda_\cA^{\infty} \bV$, or $\Lambda_\cA^{\infty} \bV_*$ (and additionally $\bL_{\fb(<)}(0^{(\infty)}|a)$ for $a\in \C$ if $m=1$), where the sequence $(b_n)$ is constant.
\end{cor}

\begin{prop}\label{prop:M.h.w}
The following statements hold:
\begin{enumerate}
\item The modules $S^\mu \bV$ and $\Pi S^\mu \bV$ (respectively, $S^\mu \bV_*$ and $\Pi S^\mu \bV_*$) are $\fb(\prec)$-highest weight modules if and only if there are $i_1,\ldots, i_k\in \Z_{<0}$ such that $i_1\prec \cdots \prec i_k\prec \Z_{<0} \setminus \{i_1,\ldots, i_k\}$ (respectively, $\Z_{<0} \setminus \{i_1,\ldots, i_k\}\prec i_k\prec \cdots \prec i_1$).
\item If either $|\{n\in \Z_{>0}\mid a_{n+1} - a_n>1\}|=\infty$ or $|\{b_n=p\}|=\infty$ for all $p\in \{\Id, \Pi\}$, then the $\g$-modules $\Lambda^{\infty}_\cA \bV$ and $\Lambda^{\infty}_\cA \bV_*$ are not highest weight modules with respect to any Borel subalgebra of $\g$. If $|\{n\in \Z_{>0}\mid a_{n+1} - a_n>1\}|<\infty$ and $|\{b_n=p\}|<\infty$ for some $p\in \{\Id, \Pi\}$, then the $\g$-module $\Lambda^{\infty}_\cA \bV$ (respectively, $\Lambda^{\infty}_\cA \bV_*$) is a $\fb(\prec)$-highest weight module if and only if $A\prec (\Z_{>0}\setminus A)$ (respectively, $(\Z_{>0}\setminus A)\prec A$).
\item The modules $S^\infty_\cA \bV$, $S^\infty_\cA \bV_*$ are not highest weight modules with respect to any Borel subalgebra of $\g$.
\end{enumerate}
\end{prop}
\begin{proof}
For an arbitrary splitting Borel subalgebra $\fb\subseteq \g$, every $\fb$-highest weight vector $v\in M$ is a $\fb_0$-singular weight vector. Now the result follows from Proposition~\ref{prop:h.w.gl(infty)}.
\end{proof}

\subsection{The case of $\fq(\infty)$} 

Let $\lambda\in \C^\infty$. Recall that $\# \lambda$ denotes the number of nonzero marks of $\lambda$, and $[a]$ denotes the greatest integer in the number $a\in \Q$.

\begin{theo}\label{thm:classification.q}
An integrable simple weight $\fq(\infty)$-module $M$ is bounded if and only if $M\cong S^\gamma \bV := \bL_{\fb(<)} (\sum_{i=1}^k \gamma_i\varepsilon_i)$ or $M\cong S^\gamma \bV_* := \bL_{\fb(>)} (\sum_{i=1}^k -\gamma_i\varepsilon_i)$, for some partition $\pmb{\gamma} = (\gamma_1> \gamma_2> \cdots > \gamma_k)$. Moreover, $S^\gamma \bV\cong \Pi S^\gamma \bV$ and $S^\gamma \bV_*\cong \Pi S^\gamma \bV_*$ if and only if $k$ is odd.
\end{theo}
\begin{proof}
Notice that $S^\gamma \bV$ (respectively, $S^\gamma \bV_*$) is bounded as it is a submodule of the bounded module $\bigotimes_{i=1}^k S^{\gamma_i} \bV$ (respectively, $\bigotimes_{i=1}^k S^{\gamma_i} \bV_*$). This proves one direction of the statement. For the other direction, note that the dimension formula for the weight spaces of $M$ from Section 2.3 shows that the number of nonzero marks of the weights of $M$ is bounded by some $l>0$. This implies that for any $i$, $M(i)\cong S^{\mu_i} V$ or $M(i)\cong S^{\mu_i} V_*$ for appropriate $\mu_i$. Fix $i_0$ and assume that $ M(i_0) \cong S^{\mu_0} V$. Let $v_{\mu_0}$ be a $\fb(<)_0$-highest weight vector of $M(i_0)$. Pick a $\fb(<_l)$-singular vector $w_0$ in $\bU(\fb(<_l)) v_{\mu_0} = \bU(\fb(<_l)_1) v_{\mu_0}$. Then $\fb(<_l) w_0 =0$, and $\g^{\varepsilon_i - \varepsilon_{j}}w_0=0$ for all $i>0$ and $j > l$, which implies that $w_0$ is a $\fb(<)$-highest weight vector of $M$. Since $M$ is simple, this shows the existence of isomorphism $M\cong \bL_{\fb(<)} (\sum \gamma_i\varepsilon_i)$ for some partition $\gamma_1 > \gamma_2 > \cdots >\gamma_k$ given by the weight of $w_0$. The strict inequality $\gamma_i>\gamma_{i+1}$ follows from the fact that $\gamma_i = \gamma_{i+1}$ implies that the simple $\fq(2)$-module $\bL_{\fb(<_2)}(\gamma_i, \gamma_{i+1})$ generated by $w_0$ is infinite dimensional \cite{Pen86}, and hence non-integrable. The case where $M(i) \cong S^{\mu_{0}} V_* $ is considered in a similar way. 

The statement that $S^\gamma \bV\cong \Pi S^\gamma \bV$ and $S^\gamma \bV_*\cong \Pi S^\gamma \bV_*$ if and only if $k$ is odd follows from \cite[Proposition~4]{Pen86}.
\end{proof}

\subsection{The remaining cases}\label{sec:remaining.cases}
Let $\g$ equal $\fosp_B(\infty|\infty)$, $\fosp_B(\infty|2k)$, $\fosp_B(m|\infty)$, $\fosp_C(2|\infty)$, $\fosp_D(\infty|\infty)$, $\fosp_D(\infty|2k)$, $\fosp_D(m|\infty)$, or $\bsp(\infty)$. In this section, $\tau$ denotes the map from the set of indices that label the standard basis of the Cartan subalgebra of $\g$ to the one-element set $\{1\}$.

Up to isomorphism, there are just two non-isomorphic spinor $\fo(2n)$-modules, $\cS_n^+$ and $\cS_n^-$, and there is a unique spinor $\fo(2n+1)$-module $\cS_n$. More precisely, consider $\cS_n^+ = L_{\fb(<_n, \tau)}(1/2,\ldots, 1/2)$, $\cS_n^- = L_{\fb(<_n, \tau)}(1/2,\ldots, 1/2,-1/2)$, and $\cS_n = L_{\fb(<_n, \tau)}(1/2,\ldots, 1/2)$.  Up to scalar, there are only two embeddings $\iota_{n}^\pm : \cS_{n-1}\hookrightarrow \cS_n$ and unique embeddings $\cS_{n-1}^+\hookrightarrow \cS_n^+$, $\cS_{n-1}^+\hookrightarrow \cS_n^-$, $\cS_{n-1}^-\hookrightarrow \cS_n^+$, and $\cS_{n-1}^-\hookrightarrow \cS_n^-$. For a given subset $A\subseteq \Z_{>0}$ we define the $\fo_B(\infty)$-module $\cS_A^B$ to be the direct limit of $\fo(2n+1)$-modules obtained from the sequence of embeddings $\{\varphi_n: \cS_{n-1}\hookrightarrow \cS_n\}$ such that $\varphi_n=\iota_n^+$ if $n\in A$ and $\varphi_n=\iota_n^-$ otherwise. In a similar way we define the $\fo_D(\infty)$-module $\cS_A^D$ to be the direct limit of $\fo(2n)$-modules obtained from the sequence of embeddings $\{\varphi_n:M_{n-1}\hookrightarrow M_n\}$ such that $M_i=\cS_i^+$ if $i\in A$ and $M_i=\cS_i^-$ otherwise. It follows from \cite[Proposition~5.3 and Theorem~5.5]{GP20} that any integrable bounded simple weight $\fo(\infty)$-module is isomorphic to $\cS_A^B$, $\cS_A^D$, or to the natural $\fo(\infty)$-module $V_\fo$.

Let $\omega_A\in \C^\infty$ be defined by setting $(\omega_A)_k =\frac{1}{2}$ if $k\in A$ and $(\omega_A)_k =-\frac{1}{2}$ otherwise. For  $A, A'\subseteq \Z_{>0}$ we write $A'\sim_B A$ if $A$ and $A'$ differ by finitely many elements, and we write $A'\sim_D A$ if $A$ and $A'$ differ by an even number of elements. By \cite[\S~5.2]{GP20}, we have $\Supp \cS_A^B=\{\omega_{A'} \in \C^{\Z_{>0}} \mid  A'\sim_B A\}$ and $\Supp \cS_A^D=\{\omega_{A'} \in \C^{\Z_{>0}} \mid A'\sim_D A\}$.

Finally, it also follows from \cite[Proposition~5.7]{GP20} that any nontrivial integrable bounded simple weight $\fsp(\infty)$-module is isomorphic to the natural $\fsp(\infty)$-module $V_\fsp$.

In Theorem~\ref{thm:classification.BCDp} below we will make use of the following remarks several times.

\begin{rem}\label{rem:supp.arguments}
\begin{enumerate}
\item Assume $\g$ equals $\fosp_B(\infty|\infty)$, $\fosp_B(\infty|2k)$, $\fosp_B(m|\infty)$, $\fosp_C(2|\infty)$, $\fosp_D(\infty|\infty)$, $\fosp_D(\infty|2k)$, $\fosp_D(m|\infty)$, or $\bsp(\infty)$. Notice that in all cases $\g_{\bar 0}\cong \fs_1\oplus \fs_2$, where $\fs_1$ is isomorphic to $\fo(\infty)$ or $\fs_2$ is isomorphic to $\fsp(\infty)$. In particular, for any constituent $M(i)$ of $M$, we have an isomorphism of (non-graded) $\g_{\bar 0}$-modules $M(i)\cong S(i)\boxtimes T(i)$, where $S(i)$ is isomorphic to an $\fs_1$-module of the form $\cS_A^B$, $\cS_A^D$, $V_\fo$ or $\C$ if $\fs_1\cong \fo(\infty)$, and $T(i)$ is isomorphic to an $\fs_2$-module of the form $V_\fsp$ or $\C$ if $\fs_2\cong \fsp(\infty)$. Since $M$ is a simple $\g$-module, any two weights of $M$ must differ from each other only by finitely many marks. This shows that once $S(i)$ is isomorphic to $V_\fo$ or $\C$, then we are not allowed to have any $S(j)$ isomorphic to $\cS_A^B$ or $\cS_A^D$. Similarly, if $S(i)$ is isomorphic to $\cS_A^B$ or $\cS_A^D$, then we are not allowed to have any $S(j)$ isomorphic to $V_\fo$ or $\C$. Also observe that if $S(i)\cong \C$ (respectively, $T(i)\cong \C$) for all $i$, then $\g_{\bar 1}M=0$. Since $\fh\subseteq [\g_{\bar 1}, \g_{\bar 1}]$, we obtain $\fh M=0$, which implies $M\cong \C$.
\item (Support arguments) Let $L$ be a weight $\g$-module, $\alpha\in \D$ be a root of $\g$, and $v\in L^\lambda$ be a nonzero weight vector. In what follows, by writing that \emph{support arguments} imply that $X_\alpha v = 0$, we mean that the vector $\alpha + \lambda\in \fh^*$ cannot lie in $\Supp L$.
\item Let $M$ be a $\fb$-highest weight $\g$-module with nonzero $\fb$-highest weight vector $v$. We define
	\[ 
|M| := \begin{cases} 
	M & \text{ if } |v| = {\bar 0} \\
	\Pi M  & \text{ if } |v|={\bar 1}.
  \end{cases}
	\]
\demo
\end{enumerate}
\end{rem}

We are now ready to state the main result of this section.

\begin{theo}\label{thm:classification.BCDp}
Let $\g$ equal $\fosp_B(\infty|\infty)$, $\fosp_B(\infty|2k)$, $\fosp_B(m|\infty)$, $\fosp_D(\infty|\infty)$, $\fosp_D(\infty|2k)$, $\fosp_D(m|\infty)$, $\fosp_C(2|\infty)$ or $\bsp(\infty)$. A nontrivial integrable simple weight $\g$-module $M$ is bounded if and only if $M\cong \bV$ or $M\cong \Pi \bV$. In particular, $M$ is locally simple.
\end{theo}
\begin{proof}
Throughout this proof $\prec$ denotes the linear order 
	\[
-1\prec 1\prec -2\prec 2\prec \cdots
	\]
on $\Z^\times $, and $A$ will be a subset of $\Z_{>0}$. The general idea is to base the proof on Lemma~\ref{lem:ss.over.Lie}, and we consider several cases in order to deal more effectively with the technical details. Since $M$ is nontrivial, Remark~\ref{rem:supp.arguments} implies that in each case below we can assume that there is at least one $S(i)$ or $T(i)$ that is not isomorphic to the trivial module $\C$.

\underline{Case $\g = \fosp_B(\infty|\infty), \fosp_D(\infty|\infty)$}. Recall that $\g_{\bar 0} \cong \fo(\infty)\oplus \fsp(\infty)$, where $\fo(\infty) = \fo_B(\infty)$ or $\fo(\infty) = \fo_D(\infty)$, respectively. Assume first that, for some $i$, there is an isomorphism of (non-graded) $\g_{\bar 0}$-modules $M(i)\cong V_\fo \boxtimes N$ for a simple bounded integrable weight $\fsp(\infty)$-module $N$. By  \cite[Proposition~5.7]{GP20} we have either $N \cong V_\fsp$ or $N\cong \C$. Suppose $N \cong V_\fsp$. Then $M(i)\cong L_{\fb(\prec,\tau)_0}(\delta_1 + \varepsilon_1)$. Moreover, support arguments imply that a $\fb(\prec,\tau)_0$-highest weight vector $v$ is also a $\fb(\prec,\tau)$-highest weight vector (see Remark~\ref{rem:supp.arguments}). In particular, $X_{\delta_2 + \varepsilon_1} v=0$.  But support arguments show also that $X_{-\delta_2 - \varepsilon_1} v=0$, and hence we get a contradiction:
	\[
0=h_{\delta_2+\varepsilon_1}v = -v.
	\]
Thus $N\cong \C$, and consequently
	\[
|M|\cong \bL_{\fb(\prec,\tau)}(\delta_1)\cong \bV.
	\]

Assume now there is an isomorphism of $\g_{\bar 0}$-modules $M(i)\cong \cS_A^B\boxtimes N$. We claim that this is not possible. Indeed, we know that $N\cong \C$ or $N\cong V_\fsp$. Suppose $N \cong V_\fsp$, and define $\sigma: \Z^\times \to \{\pm 1\}$ by setting $\sigma(j)=1$ for $j\in \Z_{>0}$, $\sigma(j)=1$ for $j\in -A$, and $\sigma(j)=-1$ otherwise. In particular, we have $M(i)\cong L_{\fb(\prec,\sigma)_0}(\omega_A + \varepsilon_1)$, and a $\fb(\prec,\sigma)_0$-highest weight vector $v$ of $M(i)$ is also a $\fb(\prec,\sigma)$-highest weight vector of $M$. Then $X_{-\delta_j - \varepsilon_1} v=0$ for any $j\notin -A$.  On the other hand, support arguments (see Remark~\ref{rem:supp.arguments}) show that $X_{\delta_j + \varepsilon_1} v=0$, and hence we get a contradiction:
	\[
0=h_{-\delta_j-\varepsilon_1}v = -v.
	\]

\underline{Case $\g = \fosp_B(m|\infty), \fosp_D(m|\infty)$}. We have $\g_{\bar 0} \cong \fo(m)\oplus \fsp(\infty)$. Assume there is an isomorphism of $\g_{\bar 0}$-modules $M(i)\cong L_{\fb(<_m, \tau)}(\lambda)\boxtimes V_\fsp$ for some weight $\lambda\in \fh(m)^*$. We claim that $\lambda=0$. Indeed, our assumption implies that $M(i)$ is a $\fb(\prec,\tau)_0$-highest weight module.  Moreover, if $v\in M(i)$ is a $\fb(\prec,\tau)_0$-highest weight vector, then support arguments (see Remark~\ref{rem:supp.arguments}) show that $X_{\delta_j + \varepsilon_2}v=0$ and $X_{-\delta_j - \varepsilon_2}v=0$ for all $j$. Thus
	\[
0=h_{\delta_j + \varepsilon_2}v=\lambda_j v,
	\]
which implies $\lambda_j=0$, and hence $\lambda=0$. 

Next we claim that $w:=X_{\delta_1 - \varepsilon_1}v\neq 0$. Indeed, support arguments imply that $X_{\varepsilon_j - \delta_{j+1}}v=0$ for $1\leq j\leq m-1$, $X_{\delta_j - \varepsilon_j}v=0$ for $2\leq j\leq m$, and $X_{\varepsilon_j - \varepsilon_{j+1}}v=0$ for $j\geq m$. Thus $w =0$ yields 
	\[
|M|\cong \bL_{\fb(\prec, \tau)}(\varepsilon_1)\cong \bL_{\fb(<,\tau)}(\varepsilon_1) \cong \varinjlim \bL_{\fb(<_n,\tau)}(\varepsilon_1).
	\]
But, by \cite[Proposition~2.3]{Kac78}, the modules $\bL_{\fb(<_n,\tau)}(\varepsilon_1)$ are not finite dimensional, and since they are simple, this is a contradiction. Thus $w\neq 0$.

Now we notice that $X_{\delta_1 - \varepsilon_1} w=0$, and again using support arguments we conclude that $X_{\varepsilon_j - \delta_{j+1}}w=0$ for $1\leq j\leq m-1$, $X_{\delta_j - \varepsilon_j}w=0$ for $2\leq j\leq m$, and $X_{\varepsilon_j - \varepsilon_{j+1}}w=0$ for $j\geq m$. In particular, $\fn(\prec,\tau) w =0$, and since the weight of $w$ is $\delta_1$ we have an isomorphism $|M| \cong \bL_{\fb(\prec, \tau)}(\delta_1)\cong \bV$ as  desired.

\underline{Case $\g = \fosp_B(\infty|2k), \fosp_D(\infty|2k)$}. Recall that $\g_{\bar 0} \cong \fo(\infty)\oplus \fsp(2k)$ where $\fo(\infty) = \fo_B(\infty)$ or $\fo(\infty) = \fo_D(\infty)$, respectively. Assume first that there exists an isomorphism of $\g_{\bar 0}$-modules  $M(i)\cong V_\fo \boxtimes N$ for some simple finite-dimensional weight $\fsp(\infty)$-module $N$. We will show that, also in this case, $|M|$ is isomorphic to $\bV$. Indeed, we have $M(i)\cong L_{\fb(\prec,\tau)_0}(\delta_1 + \sum \lambda_i\varepsilon_i)$ for some partition $\lambda = (\lambda_1\geq \cdots\geq \lambda_k)$, and support arguments imply that a $\fb(\prec,\tau)_0$-highest weight vector $v$ of $M(i)$ is also a $\fb(\prec,\tau)$-highest weight vector of $M$. Then $X_{\delta_2 + \varepsilon_1} v=0$,  and again using support arguments we obtain $X_{-\delta_2 - \varepsilon_1} v=0$. Hence 
	\[
0=h_{\delta_2+\varepsilon_1}v = -\lambda_1 v,
	\]
which implies $\lambda = 0$. Consequently, $|M|\cong \bL_{\fb(\prec,\tau)} (\delta_1)\cong \bV$.

Assume now there is an isomorphism of $\g_{\bar 0}$-modules $M(i)\cong \cS_A^B\boxtimes N$ for some simple finite-dimensional weight $\fsp(\infty)$-module $N$. We claim that this cannot happen. Recall the map $\sigma$ and the weight $\omega_A$ from case 1 above. Then we have an isomorphism of $\g_{\bar 0}$-modules $M(i)\cong L_{\fb(\prec,\sigma)_0}(\omega_A + \sum \lambda_i\varepsilon_i)$ for some partition $\lambda$. Support arguments imply that a $\fb(\prec,\sigma)_0$-highest weight vector $v$ of $M(i)$ is also a $\fb(\prec,\sigma)$-highest weight vector of $M$. Hence
	\[
|M|\cong\varinjlim \bL_{\fb(<_n,\tau)}(\nu(n) + \sum \lambda_i\varepsilon_i),
	\]
where $\nu(n)$ is a half-integer weight for every $n$. In particular, $\bL_{\fb(<_n,\tau)}(\nu(n) + \sum \lambda_i\varepsilon_i)$ is a $\g(n)$-submodule of $|M|$ for any $n$ larger than the length of the partition $\lambda$. But a necessary condition for $\bL_{\fb(<_n,\tau)}(\nu(n) + \sum \lambda_i\varepsilon_i)$ to be finite dimensional is $\lambda_k\geq n$ (see \cite[Proposition~2.3]{Kac78}). Since $\lambda$ is a finite partition and $n\to \infty$, this yields a contradiction as desired.

\underline{Case $\g = \fosp_B(2|\infty)$.} Recall that $\g_{\bar 0} \cong \C \oplus \fsp(\infty)$. Suppose that for some $i$ there is an isomorphism of $\g_{\bar 0}$-modules $M(i)\cong \C_{c\delta_1}\boxtimes V_\fsp$, where $\C_{c\delta_1}$ is a $1$-dimensional $\C$-module of weight $c\delta_1$. In other words, we have $M(i)\cong L_{\fb(<,\tau)_0}(c\delta_1 + \varepsilon_1)$. Let $v$ be a $\fb(<,\tau)_0$-highest weight vector of $M(i)$. Then $X_{\delta_1 - \varepsilon_1}v = 0$ or $X_{\delta_1 - \varepsilon_1}v = w\neq 0$. In the former case, $v$ is a $\fb(<,\tau)$-highest weight vector of $M$, and $M\cong \bL_{\fb(<,\tau)}(c\delta_1 + \varepsilon_1)$. In the latter case, $w$ a $\fb(<,\tau)$-highest weight vector of $M$, and $|M|\cong \bL_{\fb(<,\tau)}((c+1)\delta_1)$. 

Let's prove that an isomorphism $|M|\cong \bL_{\fb(<,\tau)}(c\delta_1 + \varepsilon_1)$ is contradictory. Our argument relies on some material reviewed in Section~\ref{sec:Kac-exts} below. Consider the Kac module $K(c\delta_1 + \varepsilon_1)$ and notice that there is a canonical surjective homomorphism $K(c\delta_1 + \varepsilon_1)\to \bL_{\fb(<,\tau)}(c\delta_1 + \varepsilon_1)$ which is an isomorphism whenever $\bL_{\fb(<,\tau)}(c\delta_1 + \varepsilon_1)$ is typical. Since $K(c\delta_1 + \varepsilon_1)$ is not a bounded $\g$-module (in fact, this module does not have finite-dimensional weight spaces), we obtain that $\bL_{\fb(<,\tau)}(c\delta_1 + \varepsilon_1)$ has to be atypical. This means that $c\in \{-1,1,2,\ldots \}$. Then $X_{-\delta + \varepsilon_2}v\neq 0$, since otherwise 
	\[
0=h_{\delta - \varepsilon_2}v=-cv,
	\]
which is a contradiction. Thus $(c|1,1,0\ldots)$ is a weight of $\bL_{\fb(<,\tau)}(c\delta_1 + \varepsilon_1)$, and support arguments imply that $\bL_{\fb(<,\tau)}(c\delta_1 + \varepsilon_1)$ is not bounded. 

Next we consider the case where $|M|\cong \bL_{\fb(<,\tau)}(c\delta_1)$. Again, since the nontrivial $\g$-module $\bL_{\fb(<,\tau)}(c\delta_1)$ must be atypical, we have $c \in \Z_{\geq 1}$. We claim that $c=1$. Indeed, if $c\in \Z_{\geq 2}$ then $w=X_{-\delta - \varepsilon_1}v\neq 0$, since $h_{\delta + \varepsilon_1}v=cv\neq 0$. If $X_{\delta + \varepsilon_2}w \neq 0$, then $(2|-1,1,0,\ldots)$ is a weight of $\bL_{\fb(<,\tau)}(c\delta_1)$, and support arguments show that $\bL_{\fb(<,\tau)}(c\delta_1)$ is not bounded. If $X_{\delta + \varepsilon_2}w = 0$, then $X_{-\delta - \varepsilon_2}w \neq 0$ (since $h_{\delta + \varepsilon_2}w = w\neq 0$) and $(0|-1,-1,0,\ldots)$ is a weight of $\bL_{\fb(<,\tau)}(c\delta_1)$. Again, if this is so, support arguments imply that $\bL_{\fb(<,\tau)}(c\delta_1)$ is not bounded. Therefore, $c=1$ and $|M| \cong \bL_{\fb(<,\tau)}(\delta_1)\cong \bV$.

\underline{Case $\g = \bsp(\infty)$.} Recall that $\g_{\bar 0}\cong \fsl(\infty)$.  Suppose first that, for some $i$, there is an isomorphism of $\g_{\bar 0}$-modules $M(i)\cong S^\mu V$ for a partition $\mu = (\mu_1\geq \cdots\geq \mu_k)$. Let $v_0 \in S^\mu V$ be a $\fb(<)_{\bar 0}$-highest weight vector of $M(i)$, and let $u\in \bU(\g_1)$ be a longest monomial of the form $\cdots X_{\varepsilon_2+\varepsilon_3}^{t_4}X_{2\varepsilon_2}^{t_3}X_{\varepsilon_1+\varepsilon_2}^{t_2}X_{2\varepsilon_1}^{t_1}$ with $t_i\in \{0,1\}$ such that $uv_0\neq 0$. Such a monomial exists since the vectors of the form $uv_0$ lie in $\fsl(\infty)$-submodules of $M$ isomorphic (up to parity) to $S^\nu V$ for certain partitions $\nu$, where the length of $\nu$ grows along with the length of the monomial. Thus, the non-existence of a monomial $u$ of maximal length with $uv_0=0$ would imply that $M$ is not bounded. Notice that $uv_0$ is a $\fb(<)$-highest weight vector of $M$, and hence we have an isomorphism of $\g$-modules $|M|\cong \bL_{\fb(<)}(\sum_{j=1}^\ell \gamma_i\varepsilon_j)$ for some $\gamma \in \C^\infty$ such that $\gamma_1\geq \gamma_2\geq \ldots \geq \gamma_\ell$.

We claim that $\gamma_j=0$ for all $j\geq 2$. Indeed, let $j\gg 0$ such that $\gamma_j=0$. Then, since $v$ is $\fb(<)$-highest weight, we have $X_{\varepsilon_2 + \varepsilon_j}v=0$. On the other hand, support arguments show that $X_{-\varepsilon_2 - \varepsilon_j}v=0$. Thus
	\[
0 = h_{\varepsilon_2 + \varepsilon_j}v = \gamma_2 v,
	\]
which implies $\gamma_j=0$ for all $j\geq 2$. If $j=1$, then similarly we have $X_{\varepsilon_1 + \varepsilon_2}v=0$. But now $X_{-\varepsilon_1 - \varepsilon_2}v=0$ if and only if $\gamma_1\neq 1$. In other words, we have an isomorphism $|M|\cong \bL_{\fb(<)} (\varepsilon_1)\cong \bV$. 

If $M(i)\cong S^\mu V_*$, then we prove in a similar way an isomorphism $|M|\cong \bL_{\fb(>)} (-\varepsilon_1)\cong \bV$.

Next we assume that there is an isomorphism of $\g_{\bar 0}$-modules $M(i)\cong \Lambda^{\frac{\infty}{2}}_A V$ for some $i$. Let $\prec$ be a linear order on $\Z_{>0}$ satisfying the following conditions: $A\prec (\Z_{>0}\setminus A)$, and for any $i,j\in A$ (respectively, $i,j\in \Z_{>0}\setminus A$) we have $|\{p\in A \mid i\prec p\prec j\}|<\infty$ (respectively, $|\{p\in \Z_{>0}\setminus A \mid i\prec p\prec j\}|<\infty$). Therefore we can write $\Z_{>0} = \{ j_{n_1}\prec j_{n_2}\prec \cdots \prec j_{N_2}\prec j_{N_1} \}$, where $A = \{ j_{n_1}\prec j_{n_2}\prec \cdots \}$ and $\Z_{>0}\setminus A = \{\cdots \prec j_{N_2}\prec j_{N_1}\}$. Let $\tau:\Z_{>0}\to \{1\}$, and let $v\in \Lambda^{\frac{\infty}{2}}_A V$ be a $\fb(\prec, \tau)_0$-highest weight vector. In particular, the weight of $v$ is $\varepsilon_A:=\sum_{j\in A}\varepsilon_j$. Since $X_{2\varepsilon_{i_{n_1}}}$ is a $\fb(\prec, \tau)_0$-highest weight vector of $\g_1$, we must have $w=X_{2\varepsilon_{i_{n_1}}}v=0$, as otherwise $w$ would be a $\fb(\prec, \tau)_0$-singular vector of $M$ of weight  $3\varepsilon_{i_{n_1}} +\varepsilon_{A\setminus \{i_{n_1}\}}$, which is a contradiction, as $\fsl(\infty)$ does not admit any simple bounded highest weight module with such a weight. Similarly, we must also have $X_{-\varepsilon_{i_{N_1}} -\varepsilon_{i_{N_2}}}v=0$. 

Take now $n\gg 0$ so that $j_{n_1}, j_{N_1}\in [1,n]$. Using that $X_{2\varepsilon_{j_{n_1}}}v=0$, and that $\fsl(\infty)$ does not admit a simple bounded integrable highest weight module with highest weight $2\varepsilon_{j_{n_1}} + 2\varepsilon_{j_{n_2}} + \varepsilon_{A\setminus \{j_{n_1}, j_{n_2}\}}$, we obtain that $X_{\varepsilon_{j_{n_1}}+\varepsilon_{j_{n_2}}}v=0$. Continuing this way, one shows that
	\[
X_{2\varepsilon_{j}}v=X_{\varepsilon_{j_t}+\varepsilon_{j_{t+1}}}v=0, \text{ for every } j,t\in [1,n].
	\]
On the other hand, for $j_m, j_{m+1}\in [1,n]$ such that $j_m\in A$ and $j_{m+1}\notin A$, we can use support arguments to obtain that $X_{-\varepsilon_{j_m} - \varepsilon_{j_{m+1}}}v=0$. Thus we have proved that $X_{\pm (\varepsilon_{j_{m}}+\varepsilon_{j_{m+1}})}v=0$. Since $j_{m+1}\notin A$, this yields the following contradiction
	\[
0 = h_{\varepsilon_{j_{m}}-\varepsilon_{j_{m+1}}} v = -v.
	\]
In conclusion, the isomorphism of $\g_{\bar 0}$-modules $M(i)\cong \Lambda^{\frac{\infty}{2}}_A V$ is contradictory.

Finally, assume that, for some $i$, we have an isomorphism of $\g_{\bar 0}$-modules $M(i)\cong S^\infty_A V$ for an infinite set $A = \{a_1\leq a_2\leq \cdots\}\subseteq \Z_{>0}$. For $n\gg 0$, let $w_n\in M(i)$ denote the equivalence class a $\fb(<_n)_0$-highest weight vector of a $\g(n)_{\bar 0}$-submodule of $M(i)$ isomorphic to $L_{\fb(<_n)_0}(a_n\varepsilon_1 )$. Consider $W = \bU(\g(n))w_n$ and let $w\in W$ be a $\fb(<_n)$-singular weight vector of $W$. In particular, $w$ is a $\fb(<_n)_0$-singular weight vector, and hence, it must have weight of the form $b_n\varepsilon_1$ for some $b_n\geq a_n$. Thus $X_{\varepsilon_1 + \varepsilon_2}w = 0$, and support arguments imply $X_{-\varepsilon_1 - \varepsilon_2}w = 0$. But this yields a contradiction
	\[
0=h_{\varepsilon_1 - \varepsilon_2}w = b_n w.
	\]

A similar argument shows that an isomorphism of $\g_{\bar 0}$-modules $M(i)\cong S^\infty_A V_*$ is also contradictory.
\end{proof}

\section{The category $\cB^{\Int}$}\label{sec:blocks}
Let $\cB^{\Int}$ denote the full subcategory of $\g$-mod whose objects are integrable bounded weight $\g$-modules. 

\subsection{The case $\g\ncong \fsl(\infty|1)$}
\begin{theo}\label{thm:B.C.D.sp.q-ss}
Let $\g$ equal $\fsl(\infty | m)$ with $m\in \{\Z_{>1}, \infty\}$, $\fosp_B(\infty|\infty)$, $\fosp_B(\infty|2k)$, $\fosp_B(m|\infty)$, $\fosp_C(2|\infty)$, $\fosp_D(\infty|\infty)$, $\fosp_D(\infty|2k)$, $\fosp_D(m|\infty)$, or $\bsp(\infty)$. Then the category $\cB^{\Int}$ is semisimple.
\end{theo}
\begin{proof}
Let $\g = \fsl(\infty | m)$ with $m\in \{\Z_{>1}, \infty\}$ and let $M$ and $N$ be two simple objects in $\cB^{\Int}$. By Theorem~\ref{thm:classification.sl.case}, $M \cong \varinjlim M_n$ and $N \cong \varinjlim N_n$ are locally simple. Since $M$ and $N$ are isomorphic to modules appearing in {\rm \hyperref[sym.type.mod]{($\Omega_1'$)}-\hyperref[partition.type.mod.dual]{($\Omega_6'$)}}, Lemma~\ref{lem:not.hw.trivial.ext} implies that $\Ext_{\g(n), \fh(n)}^1(M_n,N_n)=0$ for $n\gg 0$. Now the claim follows from Corollary~\ref{prop:local.triv.ext=global}.

If $\g\ncong \fsl(\infty|m)$, the result follows from Theorem~\ref{thm:classification.BCDp} and Corollary~\ref{prop:local.triv.ext=global} by noting that all $\Ext$s between the modules $\bV_n$, $\Pi \bV_n$, $\C$ or $\Pi\C$ vanish for all $n$.
\end{proof}

\begin{theo}\label{thm:B.C.D.sp.q-ss}
	If $\g = \fq(\infty)$ and $M$ and $N$ are two non-isomorphic objects of $\cB^{\Int}$, then $\Ext_{\g, \fh}^1 (M, N) = 0$ and \[ \Ext_{\g, \fh}^1 (M, M) = \begin{cases}
	0 \text{ if }M\not\cong\Pi M  \\
	\mathbb{C} \text{ if }M\cong\Pi M
	\end{cases}. \]
\end{theo}
\begin{proof}
	Recall from Theorem~\ref{thm:classification.q} that any integrable bounded simple weight $\g$-module is isomorphic to
	\[
	\bL_{\fb(<)} (\sum_{i=1}^k \gamma_i\varepsilon_i)\cong \varinjlim \bL_{\fb(<_n)} (\sum_{i=1}^k \gamma_i\varepsilon_i) \text{ or } \bL_{\fb(>)} (\sum_{i=1}^k -\gamma_i\varepsilon_i)\cong \varinjlim \bL_{\fb(>_n)} (\sum_{i=1}^k -\gamma_i\varepsilon_i)
	\]
	for some partition $\gamma = (\gamma_1> \gamma_2> \cdots > \gamma_k)$. Let $v_M$ and $v_N$ be the respective highest weight vectors of $M$ and $N$. Then the $\fq(n)$-modules $\bU(\fq(n))v_M$ and $\bU(\fq(n))v_N$ for $n\gg 0$ have different central characters. This follows from A. Sergeev's description \cite{Ser83} of the center of $\bU(\fq(n))$. Corollary~\ref{prop:local.triv.ext=global} in the Appendix implies now $\Ext_{\g, \fh}^1 (M, N) = 0$.
	
Our statement about $\Ext_{\g, \fh}^1 (M, M)$ follows directly from \cite{GS20}. There the authors consider the case of $\fq(n)$ but present an argument that extensions over $\fq(n)$ extend to $\fq(n+1)$, i.e., in fact prove that 
	\[
\Ext_{\g, \fh}^1 (M, M) = \begin{cases}
0 \text{ if }M\not\cong\Pi M  \\
\mathbb{C} \text{ if }M\cong\Pi M
\end{cases}. \qedhere
	\]
\end{proof}


\subsection{Kac modules and the case $\g=\fsl(\infty|1)$}\label{sec:Kac-exts}
We start by recalling the definition of Kac module. Assume $\g$ equals $\fsl(\infty|m)$ for $m\in \Z_{\geq 1}\cup\{\infty\}$, or $\fosp_C(2|\infty)$. Put $\g_+ := \bigoplus_{i\geq 0} \g_i$ and $\g_{\gtrless} := \bigoplus_{i\gtrless 0} \g_i$, where $\g_i$ is defined in Section~\ref{sec:prel}. Let $L$ be a simple weight $\g_{0}$-module. Set $\g_> L=0$. The \emph{Kac module} (cf. \cite{Kac78}) is the induced $\g$-module
	\[
K(L):=\bU(\g)\otimes_{\bU(\g_+)} L.
	\]
The \emph{Kac module} $K_n(L_n)$ for $\g(n)$ is defined similarly. When $L\cong L_{\fb(<)_{\bar 0}}(\lambda)$, the module $K(L)$ is usually denoted by $K(\lambda)$. The module $K(L)$ is indecomposable and admits a unique maximal proper submodule $N(L)$, yielding the short exact sequence
	\[
0\to N(L) \stackrel{f}{\longrightarrow} K(L) \stackrel{g}{\longrightarrow} \bL(L)\to 0
	\]
where $\bL(L):=K(L)/N(L)$. Similarly, $K_n(L_n)$ has a unique maximal proper submodule $N_n(L_n)$, and $\bL_n(L_n) := K_n(L_n)/N_n(L_n)$.

\begin{prop}\label{prop:g_0-emb.g-emb}
Let $\phi_{n,n+1} : L_n\hookrightarrow L_{n+1}$ be an embedding of $\g(n)_0$-modules, and consider the embedding of $\g(n)$-modules $\varphi_{n,n+1}:K_n(L_n)\hookrightarrow K_{n+1}(L_{n+1})$ mapping $u\otimes v$ to $u\otimes \phi_{n,n+1}(v)$ for all $u\in \bU(\g(n))$, $v\in L_{n}$. Then $\varphi_{n,n+1}(N_{n}(L_n))\subseteq N_{n+1}(L_{n+1})$ and $\varphi_{n,n+1}$ induces an embedding of $\g(n)$-modules $\psi_{n,n+1}:\bL_n(L_n)\hookrightarrow \bL_{n+1}(L_{n+1})$.
\end{prop}
\begin{proof}
Set $N_n = N_n(L_n)$. We claim that $\bU(\g(n+1))\varphi_{n,n+1}(N_n)$ is a proper submodule of $K_{n+1}(L_{n+1})$. Indeed, $N_n\subseteq \bU(\g(n)_{-1})^+\otimes L_n$, where $\bU(\g(n)_{-1})^+$ denotes the augmentation ideal of $\bU(\g(n)_{-1})$, and hence it is clear that $\g(n+1)_{-1}\varphi_{n,n+1}(N_n)\subseteq \bU(\g(n+1)_{-1})^+\otimes L_{n+1}$. Now we show that $\g(n+1)_+\varphi_{n,n+1}(N_n)\subseteq \bU(\g(n+1)_{-1})^+\otimes L_{n+1}$. For this it is enough to prove that $X_\alpha \varphi_{n,n+1}(N_n)\subseteq \bU(\g(n+1)_{-1})^+\otimes L_{n+1}$, where $X_\alpha$ is a simple root vector of $\g(n+1)\setminus \g(n)$. Since $X_\alpha$ commutes with $\g(n)_{-1}$ we obtain  $X_\alpha \varphi_{n,n+1}(N_n) \subseteq X_\alpha \bU(\g(n)_{-1})^+\otimes \phi_{n,n+1}(L_n) \subseteq \bU(\g(n)_{-1})^+\otimes X_\alpha L_{n+1} \subseteq \bU(\g(n)_{-1})^+\otimes L_{n+1}$. Therefore, the map $\psi_{n,n+1}(v + N_n) = \varphi_{n,n+1}(v)+N_{n+1}$ defines the desired embedding.
\end{proof}

\begin{cor}\label{prop:Kac.mod.structure}
Let $L := \varinjlim L_n$ be a locally simple weight $\g_0$-module. Then  $N(L) = \varinjlim N_n(L_n)$, and $\bL(L)\cong \varinjlim_\psi \bL_n(L_n)$ where the latter limit is taken over the sequence of embeddings $\{\bL_n(L_n)\hookrightarrow \bL_{n+1}(L_{n+1})\}$ provided by Proposition~\ref{prop:g_0-emb.g-emb}. Moreover, $\bL(L)^{\g_1} = \varinjlim \bL_n(L_n)^{\g(n)_1} \cong L$.
\end{cor}
\begin{proof}
Proposition~\ref{prop:g_0-emb.g-emb} implies that the following diagram of $\g(n)$-modules is commutative
	\[
\begin{tikzcd}
0 \arrow[r] & N_n(L_n) \arrow[r, "f_n"] \arrow[d, hook, "\varphi_{n,n+1}"]  & K_n(L_n) \arrow[r, "g_n"] \arrow[d, hook, "\varphi_{n,n+1}"] & \bL_n(L_n) \arrow[r] \arrow[d, hook, "\psi_{n,n+1}"] & 0 \\
0 \arrow[r] & N_{n+1}(L_{n+1}) \arrow[r, "f_{n+1}"]  & K_{n+1}(L_{n+1}) \arrow[r, "g_{n+1}"] & \bL_{n+1}(L_{n+1}) \arrow[r] & 0.
\end{tikzcd}
	\]
Since, for every $n$, the $\g(n)$-module $N_n(L_n)$ is the unique maximal proper submodule of $K_n(L_n)$, we conclude that $N(L) = \varinjlim N_n(L_n)$ and $\bL(L)\cong \varinjlim_\psi \bL_n(L_n)$. The claim that $\bL(L)^{\g_1} = \varinjlim \bL_n(L_n)^{\g(n)_1} \cong L$ follows from the fact that $\bL_n(L_n)^{\g(n)_1}\cong L_n$.

 \details{The composition of $\g(n)$-modules $K_n(L_n)  \stackrel{\psi_n}{\hookrightarrow} K(L)\to \bL(L)$ yields an embedding of $\g(n)$-modules $\varphi_n : \bL_n(L_n) \hookrightarrow \bL(L)$ for which $\psi_n = \psi_k\circ \psi_{n,k}$ for all $k\geq n$. Thus $\bL(L)\cong \varinjlim_\psi \bL_n(L_n)$.}
\end{proof}

Observe that for $\g = \fsl(\infty|m)$ with $m\in \Z_{\geq 2}\cup\{\infty\}$ or $\g=\fosp_C(2|\infty)$, the Kac module $K(L)$ is not bounded for any choice of $L$ since $\Lambda (\g_{-1}):=\bigoplus_{k=0}^\infty \Lambda^k (\g_{-1})$ is not bounded as a $\g_0$-module (in fact, $K(L)$ does not have finite-dimensional weight spaces). 

Assume that $M = \varinjlim \bL_{\fb(<_n)}(\lambda(n))$ is a locally simple integrable $\g$-module for a given chain of embeddings of $\g(n)$-modules $\bL_{\fb(<_n)}(\lambda(n))\hookrightarrow \bL_{\fb(<_{n+1})}(\lambda(n+1))$. We call the module $M$ \emph{typical} if there exists $N\in \Z_{>0}$ for which $(\lambda(n) + \rho_n, \beta) \neq 0$ for every $\beta\in \D(n)_{1}$ and $n\geq N$, and \emph{atypical} otherwise. Suppose in addition that $L = \varinjlim L_{\fb(<_n)_{\bar 0}}(\lambda(n))$ and the embeddings $\bL_{\fb(<_n)}(\lambda(n))\hookrightarrow \bL_{\fb(<_{n+1})}(\lambda(n+1))$ are defined as in Proposition~\ref{prop:g_0-emb.g-emb}. Then if $M = \varinjlim \bL_{\fb(<_n)}(\lambda(n))$ is typical, there is an isomorphism of $\g$-modules $M \cong K(L)$. This follows from the well known fact that $K_n(\lambda(n))$ is simple whenever $(\lambda(n) + \rho_n, \beta) \neq 0$ for all $\beta\in \D(n)_{1}$.

A weight $\mu(n)\in \fh(n)^*$ is \emph{singly atypical} if $(\mu(n), \beta)=0$ for a unique pair of mutually opposite odd roots $\pm \beta\in \D(n)_1$. It is known that if $\g(n)$ equals $\fsl(m|1)$ or $\fosp(2|2n)$ and $\lambda(n)$ dominant integral, then the $\g(n)$-module $\bL_{\fb(<_n)}(\lambda(n))$ is atypical if and only if the weight $\lambda(n)+\rho_n$ is singly atypical with respect to an odd root $\alpha_n$. In the latter case the module $K_n(\lambda(n))$ has length $2$ and its maximal proper submodule  is isomorphic to $\Pi^{p_n} \bL_{\fb(<_n)}(\lambda(n)_{\alpha_n})$, where the weight $\lambda(n)_{\alpha_n}$ is obtained by subtracting from $\lambda(n)$ a sum of positive odd roots which are uniquely determined by $\lambda(n)$ (see \cite[\S~6~and~7]{VHKT90} for details). Moreover, if $\beta_1+\cdots + \beta_{k_n}$ is this sum of odd roots then $p_n = k_n$. We also notice that $\lambda(n)_{\alpha_n}$ can be obtained from $\lambda(n)$ by a legal move of weight zero (see \cite[Corollary~6.4]{MS11} where there is a typo in the statement: it should be $\lambda(f)>\lambda(g)$). Since for $\fsl(m|1)$ and $\fosp(2|2n)$ there is at most one such legal move, there is no ambiguity in defining $\lambda(n)_{\alpha_n}$ in this way.

\begin{cor}\label{cor:Kac.mod.structure}
Suppose $\g$ equals $\fsl(\infty|1)$ or $\fosp_C(2|\infty)$.  Let $L = \varinjlim L_{\fb(<_n)_0}(\lambda(n))$ be any locally simple integrable weight $\g_0$-module. Then either $N(L)=0$ and  $\bL(L)\cong K(L)$, or $N(L)\cong \varinjlim  \Pi^{p_n}\bL_{\fb(<_n)}(\lambda(n)_\alpha)$. In particular, the $\g$-module $K(L)$ is either simple or has length $2$.
\end{cor}
\begin{proof}
The statement follows from  the above discussion and Corollary~\ref{prop:Kac.mod.structure}.
\end{proof}

Let $\cC$ be the category of weight modules with finite-dimensional weight spaces over $\g$ or $\g(n)$. For any $M\in \cC$ we can consider the \emph{restricted dual} $\g$-module $M_*\in \cC$ which is defined in \eqref{eq:lower.dual}. The  functor $M\mapsto M_*$ defines a contravariant auto-equivalence of $\cC$. Next, let $\omega$ be the automorphism of $\g$ defined by taking the direct limit of the automorphisms defined in \cite[\S~5.2]{Mus12}, and let $M^\vee$ denote the $\g$-module $M_*$ with action twisted by $\omega$ (see \cite[pg. 20]{MS11}). The functor $M\to M^\vee$ is also a contravariant auto-equivalence of $\cC$, now with the additional property that $S^\vee\cong S$ for all simple modules $S\in \cC$.

We will show that, up to applying $\Pi$, the following example provides all nontrivial extensions between simple objects of $\cB^{\Int}$ for $\g=\fsl(\infty|1)$.

\begin{example}\label{ex:nontrivial.ext}
Let $\g=\fsl(\infty|1)$ and $\C = \varinjlim L_{\fb(<_n)_0}(0^{(n)}|0)$ be the trivial one-dimensional $\g_0=\fgl(\infty)$-module. For every $n$, the weight $\rho_n\in \fh(n)^*$ is singly atypical with respect to the odd root $\alpha = \delta_1-\varepsilon$, and $(0^{(n)}|0)_\alpha = -\alpha = (0^{(n-1)}, -1|1)$. Then 
	\[
N(\C)\cong \varinjlim \Pi\bL_{\fb(<_n)}(0^{(n-1)}, -1|1) \cong \varinjlim  \Pi\bL_{\fb(<_n)}(1^{(n-1)}, 0|0) \cong \varinjlim \Pi \Lambda^{n-1}\bV_n,
	\] 
and in the category of bounded weight modules over $\fsl(\infty|1)$ we have the following non-split short exact sequence 
	\[
0\to \Lambda_\cA^\infty \bV \to K(\C)\to \C\to 0,
	\]
where $\cA $ is the sequence of ordered pairs $(a_n = n-1, b_n = 1)$ for all $n\in \Z_{>1}$.  Application of $(\cdot)_*$ on this short exact sequence yields the non-split short exact sequence
	\[
0\to \C \to K(\C)_*\to \Lambda_\cA^\infty \bV_*\to 0,
	\]
where $\Lambda_\cA^\infty \bV_*\cong \varinjlim \Pi\Lambda^{n-1}\bV_n^* \cong \varinjlim \bL_{\fb(<_n)}(0^{(n)}|1-n)$.

Set $\lambda(n):=(0^{(n)}|1-n)$, $\mu(n):=(-1^{(n)}|1)$ and $\nu(n):=(-1^{(n-1)},-2|2)$ (here we choose representatives of the weights defining $\Lambda_\cA^\infty \bV_*$, $\C$ and $\Lambda_\cA^\infty \bV$, respectively, so that the action of the center of $\fgl(n|1)$ on the modules $\bL_{\fgl}(\lambda(n))$, $\bL_{\fgl}(\mu(n))$ and $\bL_{\fgl}(\nu(n))$ coincides, see Remark~\ref{rem:reduction.to.gl(n|m)}). Then we have a sequence of legal moves of weight zero:
	\[
f_{\lambda(n)}\to f_{\mu(n)}\to f_{\nu(n)}.
	\]
Moreover, we can check that if $\gamma(n)$ is a weight such that $f_{\gamma(n)}\to f_{\lambda(n)}$ or $f_{\nu(n)}\to f_{\gamma(n)}$, then $\varinjlim \bL_{\fb(<_n)}(\gamma(n))$ is not an object in $\cB^{\Int}${\details{Indeed, they give a limit (partition growing) of the form $\varinjlim \bL_{\fb(<_n)}((1-n, -n^{(n-1)}|0))$, $\varinjlim \bL_{\fb(<_n)}((2^{(n-1)},0|0))$, respectively}}. Thus the sequence $f_{\lambda(n)}\to f_{\mu(n)}\to f_{\nu(n)}$ is maximal with the property that all objects $\varinjlim \bL_{\bf(<_n)}(\lambda(n))$, $\varinjlim \bL_{\bf(<_n)}(\mu(n))$ and $\varinjlim \bL_{\bf(<_n)}(\nu(n))$ are in $\cB^{\Int}$.
\demo
\end{example}

In the following proposition we assume that $\g=\fsl(\infty|1)$. Let $L = \varinjlim L_{\fb(<_n)_0}(\lambda(n))$, $L' = \varinjlim L_{\fb(<_n)_0}(\mu(n))$ be locally simple integrable weight $\g_0$-modules and $p,q\in \{0,1\}$. Assume also that $M := \Pi^p\bL(L)$ and $N := \Pi^q \bL(L')$ have finite-dimensional weight spaces.

\begin{prop}\label{prop:non-triv.ex=kac}
If $M = \Pi^p\bL(L)$ and $N = \Pi^q \bL(L')$, then  $\dim \Ext_{\g, \fh}^1(M,N)\leq 1$. Moreover, $\dim \Ext_{\g, \fh}^1(M,N) = 1$ precisely when, for sufficient large $n$, all $\lambda(n)+\rho_n$ are  singly atypical with respect to an odd root $\alpha_n$ and $\mu(n) = \lambda(n)_{\alpha_n}$, or vice-versa. In the latter case, if $E$ is a nontrivial extension of $M$ by $N$, then either $E\cong \Pi^pK(L)$ and $N\cong \Pi^pN(L)$, or $E\cong (\Pi^qK(L'))^\vee$ and $M\cong \Pi^qN(L')$.
\end{prop}
\begin{proof}
Let $0\to N\to E\to M\to 0$ be a non-split short exact sequence. Since the category of integrable weight $\g_0$-modules with finite-dimensional weight spaces is semisimple (see Lemma~\ref{lem:ss.over.Lie}), we can regard $M^{\g_{1}}\cong L$ and $N^{\g_1}\cong L'$ as simple $\g_0$-submodules of $E$.  As $E$ is a nontrivial extension, we obtain $E = \bU(\g)L$, and we have two possibilities: (1) $\g_1 L=0$ or (2) $\g_1 L\neq 0$. 

(1): There exists a surjective map of $\g$-modules $\Pi^pK(L)\to E$.  Since Corollary~\ref{cor:Kac.mod.structure} implies that $\Pi^pK(L)$ has length $2$ precisely when for sufficiently large $n$ the weights $\lambda(n)+\rho_n$ are singly atypical with respect to odd roots $\alpha_n$ (possibly depending on $n$), we conclude that $E\cong \Pi^pK(L)$ and $\mu(n) = \lambda(n)_{\alpha_n}$.

(2): Consider the non-split exact sequence $0\to M\to E^\vee\to N\to 0$.  Then $E^\vee = \bU(\g)L'$ and support arguments imply that $\g_1 L'=0$. Indeed, first notice that $\Supp E = \Supp M\cup \Supp N$ and set $\Delta(\g_+) := \{\beta\in \D\mid \g_\beta\subseteq \g_+\}$. Now, for any fixed $\lambda\in \Supp L$, $\lambda'\in \Supp L'$ we have $\Supp M \subseteq \lambda - \Z_{\geq 0}\Delta(\g_+)$ and $\Supp N \subseteq \lambda' - \Z_{\geq 0}\Delta(\g_+)$. Since $\g_1L\neq 0$ by assumption, we have $\g_1L\cap N\neq 0$. Thus $\lambda \in \lambda' - \Z_{\geq 0}\Delta(\g_+)$, and hence $\Supp E \subseteq \lambda'- \Z_{\geq 0}\Delta(\g_+)$. Therefore $\g_1L'=0$, and as in (1) we obtain an isomorphism of $\g$-modules $E^\vee\cong \Pi^qK(L')$, from which we conclude that $E\cong (\Pi^qK(L'))^\vee$ and $\lambda(n) = \mu(n)_{\alpha_n}$ for all sufficient large $n$.
\end{proof}

Recall that two simple modules $M,N\in \cB^{\Int}$ are in the same block if and only if $M\cong N$, or there are simple modules $M=L_1, L_2, \ldots, L_k=N$ of $\cB^{\Int}$ such that $\Ext_{\cB^{\Int}}^1(L_i, L_{i+1})\neq 0$ for all $i=1,\ldots, k-1$. The block of $M\in \cB^{\Int}$ will be denoted by $[M]$. A block $[M]$ is trivial if $[M]=\{M\}$. The next result describes the blocks of simple modules in $\cB^{\Int}$.

\begin{cor}\label{cor:blocks_sl(infty|1)}
Up to application of $\Pi$, the only nontrivial block of simple modules in $\cB^{\Int}$ is $[\C] = [\Lambda_\cA^\infty \bV] = [\Lambda_\cA^\infty \bV_*] = \{\C, \Lambda_\cA^\infty \bV, \Lambda_\cA^\infty \bV_*\}$, where $\cA $ is the sequence of ordered pairs $(a_n = n-1, b_n = 1)$ for all $n\in \Z_{>1}$.
\end{cor}
\begin{proof}
Corollary~\ref{cor:classification.sl.case} implies that it is enough to compute the blocks $[\C]$, $[S^\mu \bV]$, $[S_\cA^\infty \bV]$, $[\Lambda_\cA^{\infty} \bV]$, $[\bL_{\fb(>)}(0^{(\infty)}|a)]$ and $[\bL_{\fb(<)}(0^{(\infty)}|a)]$ for $a\in \C\setminus \Z$. The other cases will follow by application of $(\cdot)_*$ and possibly $\Pi$. Since the weight $(0^{(\infty)}|a)$ for $a\in \C\setminus \Z$ is typical we conclude that $[\bL_{\fb(<)}(0^{(\infty)}|a)]=\{\bL_{\fb(<)}(0^{(\infty)}|a)\}$ and $[\bL_{\fb(>)}(0^{(\infty)}|a)]=\{\bL_{\fb(>)}(0^{(\infty)}|a)\}$. The $\g$-modules $\C$, $S^\mu \bV$, $S_\cA^\infty \bV$ and $\Lambda_\cA^{\infty} \bV$ can be obtained as respective direct limits $\varinjlim_\psi \bL_{\fb(<_n)_0}(\lambda(n))$ where the weights $\lambda(n)$ of the three latter modules are as in {\rm \hyperref[sym.type.mod]{($\Omega_1$)}, \hyperref[ext.type.mod.dual]{($\Omega_4$)}} or {\rm \hyperref[partition.type.mod]{($\Omega_5$)}}, respectively. Now, we can check: (1) for sufficiently large $n$, all weights $\lambda(n)+\rho_n$ are atypical with respect to $\alpha = \delta_1-\varepsilon$, and in particular $\lambda(n)_\alpha = \lambda(n)-\alpha$; (2) if $\varinjlim_\psi \bL_{\fb(<_n)_0}(\lambda(n))\ncong \C, \Lambda_\cA^{\infty} \bV_*$ where $\cA$ is as in the statement, then for any $c\in \C$ the weights $\lambda(n)_\alpha + (c^{(n)}|-c)$ do not occur as $\fb(<_n)$-highest weights of modules in {\rm \hyperref[sym.type.mod]{($\Omega_1'$)}-\hyperref[partition.type.mod.dual]{($\Omega_6'$)}}, nor do they define the trivial module $\C$; {\details{(this is easy since we know precisely what is $\lambda(n)-\alpha$)}} (3) if $\varinjlim_\psi \bL_{\fb(<_n)_0}(\lambda(n))\ncong \C, \Lambda_\cA^{\infty} \bV$, where $\cA$ is as in the statement and if $\mu(n)$ is a sequence of weights such that $f_{\mu(n)} \to f_{\lambda(n)}$, then for any $c\in \C$ the weights $\mu(n)+(c^{(n)}|-c)$ do not occur as $\fb(<_n)$-highest weights of modules in {\rm \hyperref[sym.type.mod]{($\Omega_1'$)}-\hyperref[partition.type.mod.dual]{($\Omega_6'$)}}. {\details{(here we check case by case: for $\lambda(n)=(a_n, 0^{(n-1)}|0)$ we have $\mu(n) = (a_n,1^{(n-1)}|1-n)=(a_n+1-n, (2-n)^{(n-1)}|0)$; for $\lambda(n)=(\mu_1,\ldots, \mu_k, 0^{(n-k)}|0)$ we have $\mu(n) = (\mu_1,\ldots, \mu_k, 1^{(n-k)}|k-n)= (\mu_1+k-n,\ldots, \mu_k+k-n, (k-n+1)^{(n-k)}|0)$; for $\lambda(n)=(1^{(a_n)}, 0^{(k)}|0)$ with $a_n-n\geq 2$ we have $\mu(n) = (1^{(n)}|-2)= (-1^{(n)}|0)$; for $\lambda(n)=(1^{(n)}|a_n-n)$ with $a_n>n$ (and hence $n+1-a_n\leq 0$) we do not have such $\mu(n)$, since $\times (f_{\lambda(n)})=\emptyset$).}} Finally, (1)-(3) and Proposition~\ref{prop:non-triv.ex=kac} imply that up to application of $\Pi$ the only nontrivial extensions of simple objects in $\cB^{\Int}$ are given in Example~\ref{ex:nontrivial.ext}. The statement follows.
\end{proof}

\appendix
\section{ }\label{sec:appendix}
For every $n\in \Z_{>0}$, let $\g(n)$ be a finite-dimensional Lie superalgebra and let $\fh(n)\subseteq \g(n)_{\bar 0}$ be a fixed toral subalgebra of $\g(n)_{\bar 0}$, that is, each nonzero element of $\fh(n)$ acts semisimply on $\g(n)$ under the adjoint representation. It is well known that $\fh(n)$ is an abelian subalgebra of $\g(n)$ and that $\fh(n)$ acts semisimply on $\g(n)$ under the adjoint representation. An $\fh(n)$-\emph{weight} $\g(n)$-\emph{module} is by definition a $\g(n)$-module on which $\fh(n)$ acts semisimply.

An embedding of Lie superalgebras $\varphi:\g(n)\hookrightarrow \g(n+1)$ is an $\fh(n)$-\emph{weight embedding} if $\varphi(\fh(n))\subseteq \fh(n+1)$ and $\varphi$ maps every $\fh(n)$-weight space of $\g(n)$ into one $\fh(n+1)$-weight space of $\g(n+1)$. In this section, we assume that $\g$ is a Lie superalgebra isomorphic to the direct limit of a chain of weight embeddings $\g(n)\hookrightarrow \g(n+1)$. Although we are mainly interested in the Lie superalgebras listed in Section~\ref{sec:prel}, the class of Lie superalgebras we consider here is much more general, for instance $\g(n)$ may be a simple finite-dimensional Lie superalgebra of Cartan type.

Define
	\[
\bU^0:=C_{\bU(\g)}(\fh),\quad \bU_n^0 := \bU^0\cap \bU(\g(n))\quad\text{for every } n\in \Z_{>0}.
	\]

The following Lemma is a version of \cite[Lemma~4.2]{GP20}.

\begin{lem}\label{lem:A1}
If $M$ is a finite-dimensional simple $\bU^0$-module, then there exists $K> 0$ such that $M$ is a simple $\bU_n^0$-module for every $n>K$.
\end{lem}
\begin{proof}
The $\bU^0$-module structure on $M$ provides a sequence of maps $\phi_n:\bU_n^0 \to \End M$ such that $\im \phi_n\subseteq \im \phi_{k}$ for $k\geq n$. Since $\dim M<\infty$, there exists $K\in \N$ with $\im \phi_{n}=\im\phi_{k}$ for every $n\geq K$. The simplicity of $M$ as an $\bU^0$-module implies, via the Jacobson Density Theorem, that $\im (\phi: \bU^0 \to \End M)=\End M$. Since $\bU^0 = \bigcup_{n\geq 1}\bU_n^0$, we have $\im \phi = \bigcup_{n\geq 1}\im \phi_n = \im \phi_K$. Therefore $\im \phi_K=\End M$, and the statement is proved.
\end{proof}

Let $L = \bigoplus_{\mu\in \fh^*} L^\mu$ be an $\fh$-weight $\g$-module with finite-dimensional $\fh$-weight spaces $L^\mu$. Define
\begin{equation}\label{eq:lower.dual}
L_*:=\bigoplus_{\mu\in \fh^*} (L^\mu)^* \subseteq L^*.
\end{equation}
Then for any $\alpha\in \Supp \g$, $x\in \g^\alpha$, and $\lambda\in\Supp L$, we have $x(L^\lambda)^*\subseteq (L^{\lambda+\alpha})^*$. Therefore $L_*$ is an $\fh$-weight $\g$-submodule of $L^*$.

In what follows we consider the extension groups $\Ext_{\g, \fh}^i (M,N)$ in the category of $\fh$-weight $\g$-modules (see for instance \cite{Fuk86} and also \cite{Mus12}). 

The following proposition is due to V. Serganova.

\begin{prop}\label{prop:ext.homology}
Assume that $M$ and $L$ are $\fh$-weight $\g$-modules and that $L$ has  finite-dimensional weight spaces. Then $\Ext_{\g, \fh}^i (M,L) = (H_i(\g, \fh; M\otimes L_*))^*$ for any $i\in \Z_{\geq 0}$.
\end{prop}
\begin{proof}
Since $\dim L^\mu < \infty$ for every weight $\mu$, we have
\begin{equation} \label{eq:coho.homo}
\begin{split}
\Hom_\fh (M, L) & = \Hom_\fh \left(\bigoplus_\lambda M^\lambda, \bigoplus_\mu L^\mu\right)  \\
& = \prod_\lambda ((M^\lambda)^*\otimes L^\lambda) = \left(\bigoplus_\lambda M^\lambda\otimes (L^\lambda)^*\right)^* = \left( (M\otimes L_*)^\fh\right)^*,
\end{split}
\end{equation}

\noindent where $\Hom_\fh$ stands for parity preserving homomorphisms of $\fh$-modules. The statement now follows from to the fact that $\Ext_{\g, \fh}^i (M,L) := H^i(\g,\fh;\Hom_\C(M,L))$ can be computed through the cochain complex
	\[
C^i := \Hom_\fh(\Lambda^i(\g/\fh)\otimes M, L)\cong \left( (\Lambda^i(\g/\fh)\otimes M \otimes L_*)^\fh\right)^* = C_i^*,
	\]
$C_i$ being the chain complex computing the relative homology $H_i(\g,\fh; M\otimes L_*)$.
\end{proof}

\begin{cor}\label{prop:local.triv.ext=global}
Let $M = \varinjlim M_n$ and $L = \varinjlim L_n$ be $\fh$-weight $\g$-modules, and assume that $L$ has finite-dimensional $\fh$-weight spaces. If $\Ext_{\g(n), \fh(n)}^i(M_n, L_n) = 0$ for all $n\gg 0$ then $\Ext_{\g, \fh}^i(M, L) = 0$.
\end{cor}
\begin{proof}
This follows directly from Proposition~\ref{prop:ext.homology}:
\begin{align*}
\Ext_{\g, \fh}^i(M, L)  & = (H_i(\g, \fh; M\otimes L_*))^*  \\
& =  (\varinjlim H_i(\g(n), \fh(n); M_n\otimes L_n^*))^* \\
& = \varprojlim (H_i(\g(n), \fh(n); M_n\otimes L_n^*)^*) \\
& =\varprojlim \Ext_{\g(n), \fh(n)}^i(M_n, L_n)=0. \qedhere
\end{align*}
\end{proof}

The following result reproves \cite[Theorem~3.7]{PS11}.

\begin{cor}\label{cor:local.triv.ext=global.lie.algebra}
Let $\g$ equal a direct limit of finite-dimensional semisimple Lie algebras. If $M = \varinjlim M_n$ and $L = \varinjlim L_n$, where $M_n$ and $L_n$ are finite-dimensional $\fh(n)$-weight $\g(n)$-modules and $L$ has finite-dimensional $\fh$-weight spaces, then $\Ext_{\g, \fh}^1(M, L) = 0$.
\end{cor}

\begin{rem}
If $\g = \fosp(1|\infty)$, then Corollary~\ref{cor:local.triv.ext=global.lie.algebra} also holds, since the category of finite-dimensional $\fosp(1|2n)$-modules is semisimple for all $n\in \Z_{>0}$.
\demo
\end{rem}

\begin{rem}
We would like to point out also that Corollary~\ref{prop:local.triv.ext=global} does not hold without the assumption of finite-dimensionality of weight spaces. For instance, 
	\[
\Ext_{\mathbb{T}_{\fsl(\infty)}}^1(\C, \fsl(\infty))\neq 0
	\]
where $\mathbb{T}_{\fsl(\infty)}$ is the category of $\fsl(\infty)$-modules studied in \cite{PS11, D-CPS16, PSty11}.
\demo
\end{rem}


\bibliographystyle{alpha}
\bibliography{/Users/lucas/Dropbox/Research/Bib/bibliography.bib}

\end{document}